\documentclass{article}
\usepackage{geometry}                
\topmargin \dimexpr 0.65in -0.89in
\usepackage{graphicx}
\usepackage{amssymb}
\usepackage{amsthm}
\usepackage{amsmath}
\usepackage{epstopdf}
\usepackage{mathrsfs}
\usepackage{color}
\usepackage{hyperref}
\allowdisplaybreaks[2]

\newtheorem{lemma}{Lemma}[section]
\newtheorem{theorem}{Theorem}[section]
\newtheorem{definition}{Definition}[section]
\newtheorem{proposition}{Proposition}[section]

\def\bx{\mathbf{x}}
\def\bi{\mathbf{i}}
\def\bj{\mathbf{j}}
\def\b0{\mathbf{0}}
\def\bk{\mathbf{k}}
\def\ri{\iota}
\def\bbz{\mathbb{Z}}
\def\bka{\boldsymbol{\kappa}}

\title{\bf A discrete Perfectly Matched Layer for peridynamic scalar waves in two-dimensional viscous media}

\author{Yu Du\thanks{School of Mathematics and Computational Science, and Hunan Key Laboratory for Computation and Simulation in Science and Engineering, Xiangtan University, Hunan, 411105, China({\tt duyu@xtu.edu.cn})}
\and Yonglin Li\thanks{School of Mathematics and Statistics, and Hubei Key Laboratory of Computational Science, Wuhan University, Wuhan 430072, China.
({\tt yonglin.li@whu.edu.cn})}  
\and Jiwei Zhang \thanks{School of Mathematics and Statistics, and Hubei Key Laboratory of Computational Science, Wuhan University, Wuhan 430072, China. ({\tt jiweizhang@whu.edu.cn})}}
\date{}                                           

\begin{document}

\maketitle

\begin{abstract}
	In this paper, we propose a discrete perfectly matched layer (PML) for the peridynamic scalar wave-type problems in viscous media. Constructing PMLs for nonlocal models is often challenging, mainly due to the fact that nonlocal operators are usually associated with various kernels. We first convert the continua model to a spatial semi-discretized version by adopting quadrature-based finite difference scheme, and then derive the PML equations from the semi-discretized equations using discrete analytic continuation. The harmonic exponential fundamental solutions (plane wave modes) of the semi-discretized equations are absorbed by the PML layer without reflection and are exponentially damped. The excellent efficiency and stability of discrete PML are demonstrated in numerical tests by comparison with exact absorbing boundary conditions.
\end{abstract}

\section{Introduction}

Peridynamics (PD) is a nonlocal theory that has gained significant attention in recent years for its ability to model various problems in mechanics and physics, particularly those involving discontinuities such as cracks. The theory was first introduced by Silling in his seminal works \cite{silling2000reformulation,silling2007peridynamic}, and it has since been successfully applied to model the spontaneous formation of cracks \cite{silling2010crack,bobaru2016handbook,dahal2022evolution} and for Structural Health Monitoring (SHM) in both fluids and solids \cite{rahman2024peri}.
Unlike classical continuum mechanics (CCM), which relies on the assumption of spatial differentiability of displacement fields, peridynamics employs an integral formulation of force density. This nonlocal approach is particularly advantageous for modeling crack propagation, as it avoids the singularity issue encountered at the crack tip in local theories. In CCM, the stress at the crack tip is predicted to be infinite, which is not physically realistic for real materials \cite{madenci2013peridynamic}. Peridynamics, on the other hand, provides a more accurate representation of the stress distribution around the crack tip, making it a powerful tool for analyzing fracture mechanics and other discontinuous phenomena.



Modeling the propagation of elastic waves in unbounded and semi-unbounded media is indeed crucial in various fields of engineering and natural science. It is extensively applied in the analysis of soil-structure interactions and in Structural Health Monitoring \cite{lee2011consistent,tadeu2007wave,kirby2013scattering,rahman2024peri}, among other areas.
Peridynamics, while having attracted significant interest over the past two decades, still lags behind CCM models in its application to unbounded-domain problems. One of the main challenges associated with unbounded-domain problems is the need for appropriate boundary conditions. Researchers typically employ artificial boundary conditions (ABCs) to address these challenges. Since their initial introduction by Engquist \cite{engquist1977absorbing}, ABCs have been extensively investigated. However, extending these to nonlocal models like PD is problematic due to the nonlocal interactions between material points and the associated nonlocal operators, which necessitate volume-constrained boundary conditions \cite{aksoylu2011variational,seleson2013interface,mossaiby2023multi}.
Recent scholarly efforts have been focused on adapting ABCs for use with nonlocal models. This involves developing new methods and techniques to account for the nonlocal nature of PD while maintaining the effectiveness of ABCs in simulating unbounded domains. By addressing these challenges, researchers aim to enhance the applicability of PD in modeling elastic wave propagation and other related problems in unbounded and semi-unbounded media.

More recently, a number of studies have been conducted on the application of ABCs to nonlocal models. In one dimension, 
ABCs have been studied for the PD scalar wave-type equation in \cite{du2018nonlocal,ji2021accurate,wang2017transmitting}, for the nonlocal Schr{\"o}dinger equation in \cite{ji2020artificial,yan2020numerical}, and the nonlocal diffusion equation in \cite{zhang2017artificial,ZhengHuDuZhang}.
In two dimensions, \cite{DuHanZhangZheng,pang2022accurate,ma2022non} develop ABCs for the PD scalar wave-type equation. 
{Particularly, the authors of \cite{pang2022accurate,pang2023accurate,ji2021accurate} have recently proposed accurate ABCs for solving nonlocal models including peridynamics, which can be combined with spatial discrete schemes of arbitrary order.}
Although various methods have been proposed to develop ABCs for nonlocal models, recurring difficulties are evident. 
Specifically, the implementation of ABCs often leads to substantial demands on computational resources that complicate the practical application. 
This is because the ABCs typically involve the Green's functions and Fourier (Laplace) transforms (for time-domain problems), rendering the computation of kernel functions especially complicated for nonlocal models, such as PD wave-type equation, due to their inherent nonlocality, see, e.g., \cite{zhang2017artificial}.

To address this issue, we introduce the PML method for PD scalar wave equation.
The PML method, which was originally proposed by B\'{e}renger \cite{Berenger1994}, is a numerical technique for describing wave propagation in unbounded media through the imposition of an absorbing layer, and has been well developed for local problems, e.g., \cite{becache2004perfectly,bermudez2007an,cw,cl03,collino1998the,turkel1998absorbing,li2018fem,li2023new,Wang2024ANovelPML}. Unlike ABCs that introduce boundary conditions to absorb outgoing waves at the artificial boundary, PML consists of a finite-thickness artificial-material layer surrounding the truncated portion of the physical domain. The layer effectively damps the waves while creating no reflection at the transitional interface. While preliminary work has been demonstrated with the PML method applied to peridynamics \cite{antoine2019towards,DuZhangNonlocal1,DuZhangNonlocal2,wildman2011,wildman2012a,wang2013matching,ji2020artificial,duzhang2021}, there has been rooms for improvement  because the method itself suffers from several drawbacks, such as the key concept underlying its formulation lies in extending the complex coordinate stretching approach of PMLs to the nonlocal models, which is impossible for the general nonlocal operator. Recently, \cite{chern2019reflectionless} proposed a discrete PML, which is based on the $2^{\mathrm{nd}}$-order finite difference equation and is derived directly using discrete complex analysis, aiming to reduce numerical reflections at the discrete level. \cite{vicentereflectionless2023} expanded the scope from
the standard second-order finite difference method to arbitrary high-order schemes. Since it is convenient to convert the continuous PD model \eqref{eq_PDeq2D} into a spatial semi-discretized version which is equivalent to a nonlocal lattice, the discrete PML approach introduced in \cite{chern2019reflectionless} for lattices can naturally be applied for this resulting spatial semi-discretized nonlocal operator.

The aim of this paper is to present an approach for deriving discrete PML equation for the PD scalar wave-type equation. To this end, we first convert the nonlocal model into a semi-discretized version by using the spatially asymptotic compatibility (AC)  schemes on a uniform grid. After that, we take the discrete wave equations and find their associated PML equations mimicking the continuous theory but solely in the discrete setting. The derivation particularly uses \emph{Discrete Complex Analysis}. This new approach to applying discrete PML in nonlocal models offers numerous advantages. For instance, it can be applied to general nonlocal operators whereas the classical PML methods require $\mathcal{L}$ is holomorphic. The resulting semi-discetized PML equation can absorb the discrete wave impinging on the interface between the physical and PML domains, and does not produces numerical reflections at the interface due to discretization error. This reflectionless property holds for all wavelengths, even for those at the scale of the grid size. Additionally, the new method is numerically long time stable, and remarkably simple to be implemented. 
Compared with the exact absorbing boundary conditions proposed in \cite{pang2022accurate,pang2023accurate,ji2021accurate}, the discrete PML method presented in this paper demonstrates significant improvements in both computational efficiency and stability through numerical experiments.

The paper is organized as follows. The spatial discrete scheme for the PD scalar wave-type equation is  introduced in Section~\ref{sec_fdm} and the discrete PML method is discussed in Section \ref{sec_derivedispml}. Section \ref{sec_implement} is devoted to the practical implementations of discrete PML. The effectiveness of the proposed PML method on the numerical accuracy and stability is presented in Section \ref{sec_num}.

%
%
%
%

\section{Wave equation and finite difference scheme} \label{sec_fdm}

In the literature, there are two main approaches to Peridynamics (PD): bond-based PD (BB-PD) \cite{silling2000reformulation} and state-based PD (SB-PD) \cite{silling2007peridynamic}.
This paper focuses on the peridynamic scalar wave equation, which is the simplest formulation within the PD framework. The scalar wave equation is a fundamental model that captures the essential features of wave propagation in a peridynamic medium. It serves as a starting point for understanding more complex phenomena and can be extended to the bond-based and state-based formulations in future studies.
In this section, the formulation of the peridynamic scalar wave equation is briefly introduced. Additionally, a quadrature-based finite difference scheme is presented as a numerical method for solving the equation. This scheme leverages numerical integration techniques to approximate the integral terms in the PD formulation, allowing for efficient and accurate simulations of wave propagation in peridynamic media given as 


\begin{align}
  \partial_t^2 u(\bx,t) = \mathcal{L}u(\bx,t) + f(\bx),\quad \bx\in\mathbb{R}^2,\ t>0, \label{eq_PDeq2D}
\end{align}
where $f(\bx)$ is the body force, and the nonlocal operator $\mathcal{L}$ is defined by
\begin{align}
  \mathcal{L}u(\bx) = \int_{\mathcal{H}_\bx} \gamma(\bx-\bx') \left( u(\bx',t) - u(\bx,t) \right) \mathrm d\bx'.   \label{eq_nonlocaloperator}
\end{align}
\( \mathcal{H}_\bx \), called the \emph{neighborhood} of \( \bx \), is a circular region centered at \( \bx \) with radius \( \delta \), known as the \emph{horizon}. 
For the sake of clarity, we only consider the case that the nonlocal medium is spatially homogeneous,  i.e., the kernel function \( \gamma \) only depends the bond \( \bx-\bx' \).

For simplicity, in this paper, we consider only those nonnegative kernel functions of radial type, i.e., \( \gamma(\bx-\bx')=\gamma(|\bx-\bx'|) \), such as the common Gaussian kernel (cf. \cite{DuHanZhangZheng,wildman2012a,du2018nonlocal}, etc)
\begin{align}
 	\gamma(\bx-\bx') = \frac{4}{\pi\epsilon^4} e^{-\frac{|\bx-\bx'|^2}{\epsilon^2}},\quad \epsilon>0,
\end{align}
and the unbounded kernel (cf. \cite{pang2023accurate,hermann2023dirichlet}, etc)
\begin{align}
	\gamma(\bx'-\bx) = \bar{\gamma}(|\bx'-\bx|) \frac{1}{|\bx'-\bx|^{2+2s}},\quad 0\leq s<\frac12. \label{eq_generalkernel}
\end{align}
where the scalar-valued function $\bar{\gamma}$ is often defined as either Heaviside, piecewise linear or Gaussian functions within the neighborhood $\mathcal{H}_\b0$.

We now use an asymptotic compatible quadrature-based finite difference scheme \cite{du2019asymptotically,tian2014asymptotically} to the nonlocal wave equation \eqref{eq_PDeq2D}. This scheme is very robust for numerically approximating nonlocal modes because its convergence is not sensitive to the choice of modeling and discretization parameters, in particular for sufficiently small horizons $\delta$ or sufficiently fine grid spacing.

For simplicity, in what follows, the discrete PML is stated in two dimensions, but they can be readily extended to one or three (higher) dimensions.

Let $\{\bx_\bi = h \bi\in\mathbb{R}^2 : \bi=(i_1,i_2)\in\mathbb{Z}^2\}$ be the set of uniform Cartesian mesh nodes with mesh size $h$. With the main focus being on the discrete nonlocal PML, without loss of generality, we set \( f(\bx,t)=0 \) and consider the multi-half-space problem; that is, we leave \( \{i_1<0\}\cup \{i_2<0\} \) as the physical domain, and let \( \{i_1\geq0\}\cup \{i_2\geq0\} \) be the domain for PML.

Then nonlocal operator at node $\bx_\bi$ can be rewritten as
\begin{align}
	\mathcal{L} u(\bx_\bi,t) = \int_{\mathcal{H}_{\mathbf{0}}} \frac{ u(\mathbf{z}+\bx_\bi,t) - u(\bx_\bi,t) }{ W(\mathbf{z}) } W(\mathbf{z}) \gamma(\mathbf{z}) \mathrm{d} \mathbf{z},
\end{align}
where $W(\mathbf{z})=|\mathbf{z}|^2 / |\mathbf{z}|_1$ is a weight function. Here $|\cdot|_1$ stands for the $l_1$ norm and $|\cdot|$ denotes the standard Euclidean norm.

A quadrature-based AC finite difference scheme of the nonlocal operator \eqref{eq_nonlocaloperator} is given by
\begin{align}
	\mathcal{L}_h u(\bx_\bi) = \int_{\mathcal{H}_{\mathbf{0}}} \mathcal{I}_h \left( \frac{ u(\mathbf{z}+\bx_\bi) - u(\bx_\bi) }{ W(\mathbf{z}) } \right) W(\mathbf{z}) \gamma(\mathbf{z}) \mathrm{d} \mathbf{z},\label{eq_deflh}
\end{align}
where $\mathcal{I}_h(\cdot)$ represents the piecewise bilinear interpolation operator defined as 
\begin{align}
	\mathcal{I}_h \left( \frac{ u(\mathbf{z}+\bx_\bi) - u(\bx_\bi) }{ W(\mathbf{z}) } \right) = \sum_{\mathbf{k}\in\mathbb{Z}^2} \frac{ u(\mathbf{x}_\mathbf{k}+\bx_\bi) - u(\bx_\bi) }{ W(\mathbf{x}_\mathbf{k})} \phi_{\mathbf{k}}(\mathbf{z}),\quad \mathbf{z}\in\mathbb{R}^2.
\end{align}
Here $\phi_\mathbf{k}$ is the piecewise bilinear basis function satisfying $\phi_\mathbf{k}(\bx_\bi)=0$ when $\bi\neq\mathbf{k}$ and $\phi_{\mathbf{k}}(\bx_\mathbf{k})=1$.

Setting $u_\bi(t)$ to stand for an approximation of $u(\bx,t)$ at node $\bx_\bi$, we reformulate \eqref{eq_deflh} into
\begin{align}
	\mathcal{L}_h u_\bi(t) = \sum_{\bk\in\mathbb{Z}^2} a_\mathbf{k} u_{\bi+\bk}(t),  \label{eq_deflhui}
\end{align}
where $a_\mathbf{0}=-\sum_{\mathbf{k}\neq\mathbf{0}} a_\mathbf{k}$ and 
\begin{align}
	a_{\mathbf{k}} = \frac{1}{W(\bx_{\mathbf{k}})} \int_{\mathcal{H}_{\mathbf{0}}} \phi_\mathbf{k}(\mathbf{z}) W(\mathbf{z}) \gamma(\mathbf{z}) \mathrm{d} \mathbf{z},\quad \mathbf{k}\neq\mathbf{0}.
\end{align}
Thus, the resulting quadrature-based AC finite difference scheme~\eqref{eq_diswaveeq} can be derived by substituting \( \mathcal{L}_h \) into the nonlocal wave equation~\eqref{eq_PDeq2D} as 
\begin{align}
	\partial_t^2 u_\bi(t) &= \mathcal{L}_h u_\bi(t),\quad \bi\in \bbz^2.\label{eq_diswaveeq}
\end{align}

\section{Discrete PML} \label{sec_derivedispml}
We first recall a few preliminary facts.
For some constant amplitudes \( W_{\bka,\omega} \), the radiating solutions in the physical domain generally take the form of a superposition of plane waves
\begin{align*}
 	u(\bx_\bi,t) = \sum_{\bka,\omega} W_{\bka,\omega} e^{-\ri\omega t} e^{\ri\bka\cdot\bx_\bi},
\end{align*}
 where \( \omega \) is the (angular) frequency, \( \bka \) is the wavevector and \( \ri \) is the imaginary unit. For waves propagating in the \( +x_\alpha\, (\alpha=1,2) \) direction, \( \mathrm{sgn}(\omega)\Re \kappa_\alpha \) is positive. \( \Im \kappa_\alpha \) is less than or equal to $0$ because of the boundedness assumption. In the domain for PML, the PD scalar wave-type equation has the following nonlocal dispersion relation \cite{hermann2023dirichlet}
\begin{align}
 	-\omega^2 = \int_{\mathcal{H}_\b0} \gamma(\bx) \left( e^{\ri\bka\cdot\bx} - 1 \right) \mathrm{d} \bx = \int_0^\infty \int_0^{2\pi} \gamma(r) \left( 1 - \cos(|\bka|r\cos\theta) \right) \mathrm{d}r \mathrm{d}\theta,
\end{align}
where $\bka = |\bka|(\cos\theta,\sin\theta)$.

The solution of the semi-discrete wave equation consists of elementary waves \[ u_\bi(t) = e^{-\ri\omega t}e^{\ri\boldsymbol{\kappa}^h\cdot \bi h}, \] where \( \omega\in\mathbb{R} \) and \( \bka^h=(\kappa_1^h,\kappa_2^h)\in\mathbb{C}^2 \) are linked by the discrete dispersion relation
\begin{align}
 	-\omega^2 =&~ a_0 + \sum_{k\in\bbz^+} 2a_k \left( \cos(\kappa_1^h h) + \cos(\kappa_2^h h) \right) + \sum_{\bk\in(\bbz^+)^2} 4 a_\bk \cos(\kappa_1^hh)\cos(\kappa_2^hh). \label{eq_disdeprel}
\end{align}

In the continuous theory, the PML equations are derived by using complex coordinate stretching of coordinates in the analytic continuation of
the wave equation. We take a glimpse of the classical PML for the nonlocal model via complex coordinate transform (see, e.g., \cite{antoine2019towards})
\begin{align}
 	x_\alpha \to \tilde{x}_\alpha = x_\alpha + \ri \int_0^\alpha \frac{\sigma^\alpha(s)}{\omega} \mathrm{d} s,
\end{align}
where \( \sigma^\alpha \) is the PML absorbing satisfying \( \sigma^\alpha(s)|_{s>0}>0 \) and \( \sigma^\alpha(s)|_{s\leq0}=0 \). The classical PML on Cartesian coordinates is seeking a complexification of the wave equation by combining the Fourier (Laplace) transform and the complex coordinates 
\begin{align}
 	\prod_{\alpha=1}^2 (\partial_t + \sigma^\alpha) u(\bx,t) = \mathcal{L}^{pml} \partial_t u(\bx,t) + \mathcal{L} u(\bx,0),
\end{align}
where \( \mathcal{L}^{pml} \) is the nonlocal operator for the PML (cf. \cite{du2023numerical,duperfectly}). However, classical PML has two insurmountable difficulties: First, the definition of \(\mathcal{L}^{pml} \) combines the time convolution with the inverse Laplace transformation which is numerically expensive; Second, it requires the assumption that \(\mathcal{L} \) should be a holomorphic operator along each component, which is too strict to be implemented.


To overcome these difficulties, we consider the discrete PML method proposed by Albert Chern in \cite{chern2019reflectionless}, which aims to find a ``complexification'' of the discrete equation~\eqref{eq_diswaveeq} by using discrete complex analysis. The normal modes of the discrete PML equation are the discrete analytic continuation of \( e^{\ri\boldsymbol{\kappa}^h\cdot \bi h} \), given by
\begin{align}
 	 \prod_{\alpha=1}^2 \eta(i_\alpha,\sigma^\alpha,\kappa_\alpha^h,\omega) e^{\ri \kappa_\alpha^h i_\alpha h}\quad \mbox{ with }\quad \eta(i_\alpha,\sigma^\alpha,\kappa_\alpha^h,\omega)=  \prod_{l=0}^{i_\alpha-1} \frac{ 2+\ri\frac{\sigma^\alpha_l}{\omega}(1-e^{-\ri \kappa_\alpha^h h}) }{ 2+\ri\frac{\sigma^\alpha_l}{\omega}(1-e^{\ri \kappa_\alpha^h h}) }, \label{eq_disanacondiswave}
\end{align}
where \( \eta(i_\alpha,\sigma^\alpha,\kappa_\alpha^h,\omega) \) is an approximant of \( e^{-\kappa_\alpha^h \int_0^{i_\alpha h} \frac{\sigma^\alpha}{\omega} \mathrm{d} s} \) by Pad{\'e} approximation as
\begin{align}
 	e^{-\kappa_\alpha^h \int_0^{i_\alpha h} \frac{\sigma^\alpha}{\omega} \mathrm{d} s} = \prod_{l=0}^{i_\alpha-1} e^{-\kappa_\alpha^h \frac{\sigma^\alpha_l}{\omega}h} \approx \prod_{l=0}^{i_\alpha-1} \frac{ 2-\sigma_l^\alpha h \frac{\sigma^\alpha_l}{\omega} }{ 2+\sigma_l^\alpha h \frac{\sigma^\alpha_l}{\omega} } \approx \prod_{l=1}^{i_\alpha-1} \frac{ 2+\ri\frac{\sigma^\alpha_l}{\omega}(1-e^{-\ri \kappa_\alpha^h h}) }{ 2+\ri\frac{\sigma^\alpha_l}{\omega}(1-e^{\ri \kappa_\alpha^h h}) }.
\end{align}
Here we denote by \( \sigma^\alpha_l = \sigma^\alpha(s_l) \) for some \( s_l\in( lh,(l+1)h) \) with the mean value theorem for integral. For \( \sigma_l^\alpha>0,\omega\in\mathbb{R}, \) and \( \kappa_\alpha^h\in\mathcal{K} \), it is easy to verify that
\begin{align}
 	\left| \frac{ 2+\ri\frac{\sigma^\alpha_l}{\omega}(1-e^{-\ri \kappa_\alpha^h h}) }{ 2+\ri\frac{\sigma^\alpha_l}{\omega}(1-e^{\ri \kappa_\alpha^h h})} \right| < 1.\label{est_exp}
\end{align}
 Therefore, the discrete analytic continuation~\eqref{eq_disanacondiswave} decays exponentially as \( i_\alpha\to\infty \).

In this section, we will use this idea to derive the discrete PML based on the discrete nonlocal wave equation~\eqref{eq_diswaveeq}.

\subsection{Discrete complex differentiation}
For the applications to  discrete PML, we only introduce some essential definitions to calculate the analytic continuation of discrete plane waves (cf. \cite{chern2019reflectionless,vicentereflectionless2023}). One can also refer to  \cite{duffin1956basic,bobenko2005linear,lovasz2004discrete,bobenko2016discrete} for further study. Let \( \Lambda \) be an infinite quadrilateral lattice indexed by \( \bbz^2 \) which is described by a set of vertices \( V(\Lambda)=\bbz^2 \), 
a set of faces
\begin{align*}
 	 F(\Lambda)=\{Q_{i+\frac12,j+\frac12}:=((i,j),(i+1,j),(i+1,j+1),(i,j+1))|\ (i,j)\in\bbz^2 \}.
\end{align*}

Denote by \( \mathbb{C}^{\bbz\times\bbz} \) the set of functions \( v:\bbz\times\bbz\to\mathbb{C} \), and by \( \mathbb{C}^{\bbz^2\times\bbz^2} \) the set of functions \( v:\bbz^2\times\bbz^2\to\mathbb{C} \). For any \( v\in\mathbb{C}^{\bbz\times\bbz} \), let \( v_{\bi} \) (or \( v_{i_1,i_2} \)) be its value on vertice \( \bi=(i_1,i_2)\in\bbz^2 \). For any \( v\in\mathbb{C}^{\bbz^2\times\bbz^2} \), let \( v_{\bi,\bj} \) (or \( v_{(i_1,i_2),(j_1,j_2)} \)) be its value on \( (\bi,\bj)=((i_1,i_2),(j_1,j_2))\in\bbz^2\times \bbz^2 \).

Let \( z:V(\Lambda)\to\mathbb{C} \) be a complex valued function on \( V(\Lambda) \), i.e., \( z\in\mathbb{C}^{\bbz\times\bbz} \). The lattice \( \Lambda \) with the structure \( z=(z_\bi)_{\bi\in V(\Lambda)} \) characterizes a discrete complex domain \( (\Lambda,z) \).


\begin{definition}[Discrete holomorphicity]
	A complex-valued function \( f\in\mathbb{C}^{\bbz\times\bbz} \) is said to be complex-differentiable with respect to \( z \) at a face \( Q_{i+\frac12,j+\frac12} \in F(\Lambda) \) if \( f \) satisfies the  following discrete Cauchy-Riemann equation
	\begin{align}
	 	\frac{ f_{i+1,j+1} - f_{i,j} }{ z_{i+1,j+1} - z_{i,j} } = \frac{ f_{i,j+1} - f_{i+1,j} }{ z_{i,j+1} - z_{i+1,j} }.	\label{eq_defdiscomder}
	\end{align}
And \( f \) is said to be discrete-holomorphic on \( (\Lambda,z) \) if \( f \) is complex-differentiable with respect to \( z \) at every face in \( F(\Lambda) \). The discrete complex derivative $D_zf(Q_{i+\frac12,j+\frac12})$ is defined as  \eqref{eq_defdiscomder}.
\end{definition}


Now in order to derive the discrete PML equations, for each dimensional component \( \alpha \) of the original domain \( \bbz^2 \) a discrete complex domain \( (\Lambda, z^\alpha) \) is designed, where \( z^\alpha\in\mathbb{C}^{\bbz\times\bbz} \) is given by
\begin{align}
 	z_{i,j}^\alpha = (i+j)h + \ri\frac{h}{\omega} \sum_{l=0}^{j-1} \sigma_l^\alpha,\quad i,j\in\bbz		\label{eq_defzij},
\end{align}
where \( \sigma^\alpha_l \) are the PML absorbing coefficients satisfying \( \sigma^\alpha_l \geq 0\) for all $l$ and $\sigma_l^\alpha=0$ if $l<0$.

\begin{definition}
We say that \( w_{\bi,\bj}\in\mathbb{C}^{\bbz^2\times\bbz^2} \) is discrete holomorphic on \( (\Lambda, z^\alpha) \) for each fixed \( \alpha \in\{1,2\} \) if and only if the following identities  simultaneously hold: 
\begin{align}
 	& \frac{ w_{(i_1+1,i_2),(j_1+1,j_2)} - w_{\bi,\bj} }{ z_{i_1+1,j_1+1}^1 - z_{i_1,j_1}^1 } = \frac{ w_{(i_1,i_2),(j_1+1,j_2)} - w_{(i_1+1,i_2),(j_1,j_2)} }{ z_{i_1,j_1+1}^1 - z_{i_1+1,j_1}^1 }, \label{eq_udishol1}\\
 	& \frac{ w_{(i_1,i_2+1),(j_1,j_2+1)} - w_{\bi,\bj} }{ z_{i_2+1,j_2+1}^2 - z_{i_2,j_2}^2 } = \frac{ w_{(i_1,i_2),(j_1,j_2+1)} - w_{(i_1,i_2+1),(j_1,j_2)} }{ z_{i_2,j_2+1}^2 - z_{i_2+1,j_2}^2 }. \label{eq_udishol2}
\end{align}
\end{definition}

For ease of presentation, let $\tau_k^\alpha$ be the translation operator with respect to $i_\alpha$ and \( \rho_k^\alpha \) be the translation operator with respect to $j_\alpha$,
\begin{align}
 	& \tau_k^1 w_{\bi,\bj} = w_{(i_1+k,i_2),\bj},\ \tau_k^2 w_{\bi,\bj} = w_{(i_1,i_2+k),\bj},\quad \forall w_{\bi,\bj}\in \mathbb{C}^{\bbz^2\times\bbz^2},\\
 	& \rho_k^1 w_{\bi,\bj} = w_{\bi,(j_1+k,j_2)},\ \rho_k^2 w_{\bi,\bj} = w_{\bi,(j_1,j_2+k)},\quad \forall w_{\bi,\bj}\in \mathbb{C}^{\bbz^2\times\bbz^2}.
\end{align}
Thus, \eqref{eq_udishol1} and \eqref{eq_udishol2} can be rewritten as
\begin{align}
 	\frac{ \tau_1^\alpha\rho_1^\alpha w_{\bi,\bj} - w_{\bi,\bj} }{ \tau_1^\alpha\rho_1^\alpha z_{\bi,\bj}^\alpha - z_{\bi,\bj}^\alpha } = \frac{ \rho_1^\alpha w_{\bi,\bj} - \tau_1^\alpha w_{\bi,\bj}}{ \rho_1^\alpha z_{\bi,\bj}^\alpha - \tau_1^\alpha z_{\bi,\bj}^\alpha },\quad \alpha=1,2, \label{eq_discomderw}
\end{align}
where \( z_{\bi,\bj}^\alpha\in\mathbb{C}^{\bbz^2\times\bbz^2} \) is defined by \( z_{\bi,\bj}^\alpha=z_{i_\alpha,j_\alpha}^\alpha \) and \( z_{i_\alpha,j_\alpha}^\alpha \) is given by \eqref{eq_defzij}.

Define scaled discrete complex derivative at the face $Q_{i_\alpha+\frac12,j_\alpha+\frac12}$ by
\begin{align}
	D_\alpha w_{\bi,\bj} = \frac{1}{\ri\omega} \frac{ \tau_1^\alpha\rho_1^\alpha w_{\bi,\bj} - w_{\bi,\bj} }{ \tau_1^\alpha\rho_1^\alpha z_{\bi,\bj}^\alpha - z_{\bi,\bj}^\alpha }\in \mathbb{C}^{\bbz^2\times\bbz^2}, \quad \alpha=1,2	\label{eq_scadiscomder}
\end{align}
and scaled second-order discrete complex derivative by
\begin{align}
 	D_{\alpha,\beta} w_{\bi,\bj} = D_{\alpha} \left( D_{\beta} w_{\bi,\bj} \right)\in \mathbb{C}^{\bbz^2\times\bbz^2},\quad \alpha,\beta=1,2.	\label{eq_scadiscomder2}
\end{align}
For \( w_{\bi,\bj},v_{\bi,\bj} \in \mathbb{C}^{\bbz^2\times\bbz^2} \), 
it is easy to prove that \( D_\alpha \) is a linear operator (i.e., \( D_\alpha(w_{\bi,\bj}+v_{\bi,\bj}) = D_\alpha w_{\bi,\bj} + D_\alpha v_{\bi,\bj} \)),  and 
\begin{align}
	\tau_k^\alpha (D_\beta w_{\bi,\bj})=D_\beta( \tau_k^\alpha w_{\bi,\bj} ),\quad \rho_k^\alpha (D_\beta w_{\bi,\bj}) = D_\beta ( \rho_k^\alpha w_{\bi,\bj} ),  \quad \alpha,\beta \in \{1,2\}\text{ and } k\in\bbz. 
\end{align}

Next, we give some propositions on the shift operators and discrete complex derivatives.
\begin{proposition} \label{th_da12}
 	Assume that $w_{\bi,\bj}\in\mathbb{C}^{\bbz^2\times\bbz^2}$ is discrete-holomorphic on \( (\Lambda, z^\alpha) \) for each fixed \( \alpha\in\{1,2\} \), then for \(\beta\in\{1,2\} \) and $\beta\neq\alpha$, it holds $D_{\alpha,\beta} w_{\bi,\bj} = D_{\beta,\alpha} w_{\bi,\bj}$, and moreover,
 	\begin{align}
 		D_{\alpha,\beta} w_{\bi,\bj} = \frac{1}{\ri\omega}  \frac{ \rho_1^{\alpha} (D_{\beta} w_{\bi,\bj})  - \tau_1^{\alpha} (D_{\beta} w_{\bi,\bj})}{ \rho_1^{\alpha} z_{\bi,\bj}^{\alpha} - \tau_1^{\alpha} z_{\bi,\bj}^{\alpha} }, \label{eq_dabdh}
	\end{align}
	which also indicates that  \( D_{\beta} w_{\bi,\bj}  \) is discrete-holomorphic on \( (\Lambda, z^{\alpha}) \).
\end{proposition}

\begin{proof}
	It follows from \eqref{eq_scadiscomder} and \eqref{eq_scadiscomder2} that 
	\begin{align*}
 		D_{\alpha,\beta} w_{\bi,\bj} =&~ D_{\alpha} \left( D_{\beta} w_{\bi,\bj} \right) = \frac{1}{\ri\omega} \frac{ \tau_1^{\alpha}\rho_1^{\alpha} (D_{\beta} w_{\bi,\bj}) - (D_{\beta} w_{\bi,\bj}) }{ \tau_1^{\alpha}\rho_1^{\alpha} z_{\bi,\bj}^{\alpha} - z_{\bi,\bj}^{\alpha} } \\
 		=&~ \frac{1}{(\ri\omega)^2} \frac{ \tau_1^{\alpha}\rho_1^{\alpha} (\tau_1^{\beta}\rho_1^{\beta} w_{\bi,\bj} - w_{\bi,\bj}) - (\tau_1^{\beta}\rho_1^{\beta} w_{\bi,\bj} - w_{\bi,\bj}) }{ (\tau_1^{\alpha}\rho_1^{\alpha} z_{\bi,\bj}^{\alpha} - z_{\bi,\bj}^{\alpha}) (\tau_1^{\beta}\rho_1^{\beta} z_{\bi,\bj}^{\beta} - z_{\bi,\bj}^{\beta}) } \\
 		=&~ \frac{1}{(\ri\omega)^2} \frac{ \tau_1^{\beta}\rho_1^{\beta} (\tau_1^{\alpha}\rho_1^{\alpha} w_{\bi,\bj} - w_{\bi,\bj}) - (\tau_1^{\alpha}\rho_1^{\alpha} w_{\bi,\bj} - w_{\bi,\bj}) }{ (\tau_1^{\alpha}\rho_1^{\alpha} z_{\bi,\bj}^{\alpha} - z_{\bi,\bj}^{\alpha}) (\tau_1^{\beta}\rho_1^{\beta} z_{\bi,\bj}^{\beta} - z_{\bi,\bj}^{\beta}) } \\
 		=&~ \frac{1}{\ri\omega} \frac{ \tau_1^{\beta}\rho_1^{\beta} (D_{\alpha} w_{\bi,\bj}) - (D_{\alpha} w_{\bi,\bj}) }{ (\tau_1^{\beta}\rho_1^{\beta} z_{\bi,\bj}^{\beta} - z_{\bi,\bj}^{\beta}) } \\
 		=&~ D_{\beta} \left( D_{\alpha} w_{\bi,\bj} \right) = D_{\beta,\alpha} w_{\bi,\bj}.
	\end{align*}
	Since $w_{\bi,\bj}$ is discrete-holomorphic for each \( \alpha \), it follows from \eqref{eq_discomderw} that 
	\begin{align*}
 		D_{\alpha,\beta} w_{\bi,\bj} =&~ D_{\beta} \left( D_{\alpha} w_{\bi,\bj} \right) = \frac{1}{\ri\omega} \frac{ \tau_1^{\beta}\rho_1^{\beta} (D_{\alpha} w_{\bi,\bj}) -  (D_{\alpha} w_{\bi,\bj}) }{ \tau_1^{\beta} \rho_1^{\beta} z_{\bi,\bj}^{\beta} - z_{\bi,\bj}^{\beta} } \\
 		=&~ \frac{1}{(\ri\omega)^2} \frac{ \tau_1^{\beta}\rho_1^{\beta} ( \rho_1^{\alpha} w_{\bi,\bj} - \tau_1^{\alpha} w_{\bi,\bj} ) -  ( \rho_1^{\alpha} w_{\bi,\bj} - \tau_1^{\alpha} w_{\bi,\bj} ) }{ (\tau_1^{\beta} \rho_1^{\beta} z_{\bi,\bj}^{\beta} - z_{\bi,\bj}^{\beta}) ( \rho_1^{\alpha} z_{\bi,\bj}^{\alpha} - \tau_1^{\alpha} z_{\bi,\bj}^{\alpha} ) } \\
 		=&~ \frac{1}{(\ri\omega)^2} \frac{ \rho_1^{\alpha} ( \tau_1^{\beta}\rho_1^{\beta} w_{\bi,\bj} - w_{\bi,\bj} ) -  \tau_1^{\alpha} ( \tau_1^{\beta}\rho_1^{\beta} w_{\bi,\bj} - w_{\bi,\bj} ) }{ (\tau_1^{\beta} \rho_1^{\beta} z_{\bi,\bj}^{\beta} - z_{\bi,\bj}^{\beta}) ( \rho_1^{\alpha} z_{\bi,\bj}^{\alpha} - \tau_1^{\alpha} z_{\bi,\bj}^{\alpha}) } \\
 		=&~ \frac{1}{\ri\omega}  \frac{ \rho_1^{\alpha} (D_{\beta} w_{\bi,\bj})  - \tau_1^{\alpha} (D_{\beta} w_{\bi,\bj})}{ \rho_1^{\alpha} z_{\bi,\bj}^{\alpha} - \tau_1^{\alpha} z_{\bi,\bj}^{\alpha} }.
	\end{align*}
	This completes the proof.
\end{proof}

Next, we consider the relationship between \( \tau_{-k}^\alpha,\tau_{k}^\alpha \) and \( \rho_{-k}^\alpha,\rho_{k}^\alpha \).
\begin{proposition} \label{pro_tauw}
	Assume that $w_{\bi,\bj}\in\mathbb{C}^{\bbz^2\times\bbz^2}$ is discrete-holomorphic on \( (\Lambda, z^\alpha) \) for each fixed \( \alpha \). Then we have that for \( k\in\bbz^+ \)
	\begin{align}
 		\tau_{-k}^\alpha w_{\bi,\bj} =&~ \rho_{-k}^\alpha w_{\bi,\bj} - h \sum_{l=1}^{k} \rho_{-l}^\alpha  \left(  \sigma_{\bj}^\alpha \left(	\tau_{-k+l-1}^\alpha D_\alpha  w_{\bi,\bj} \right) \right), \label{eq_taukw}\\
 		\tau_{k}^\alpha w_{\bi,\bj} =&~ \rho_{k}^\alpha w_{\bi,\bj} + h \sum_{l=1}^k \rho_{l-1}^\alpha  \left(  \sigma_{\bj}^\alpha \left( \tau_{k-l}^\alpha D_\alpha  w_{\bi,\bj} \right) \right), \label{eq_taukw2}
	\end{align}
	and when $\alpha\neq\beta$,
	\begin{align}
 		\tau_{-k}^{\alpha} D_{\beta}  w_{\bi,\bj} =&~ \rho_{-k}^{\alpha}  \left( D_{\beta} w_{\bi,\bj} \right) - h \sum_{l=1}^k   \rho_{-l}^{\alpha}  \left( \sigma_{\bj}^{\alpha} \left( \tau_{-k+l-1}^{\alpha} D_{\alpha,\beta}  w_{\bi,\bj} \right) \right), \label{eq_taukDw}\\
 		\tau_{k}^{\alpha} D_{\beta} w_{\bi,\bj} =&~ \rho_{k}^{\alpha} \left( D_{\beta}  w_{\bi,\bj} \right) + h \sum_{l=1}^k   \rho_{l-1}^{\alpha}  \left( \sigma_{\bj}^{\alpha} \left( \tau_{k-l}^{\alpha} D_{\alpha,\beta}  w_{\bi,\bj} \right) \right). \label{eq_taukDw2}
	\end{align}
	Here, $\sigma^\alpha_\bj:=\sigma_{j_\alpha}^\alpha$ and $\sigma_j^\alpha$ is defined in \eqref{eq_defzij}.
\end{proposition}

\begin{proof}
	We only prove the identities \eqref{eq_taukw} and \eqref{eq_taukDw}. The proof of \eqref{eq_taukw2} and \eqref{eq_taukDw2} is analogous.
	
	It follows from the definition of $z^\alpha$ \eqref{eq_defzij}, and \eqref{eq_discomderw}--\eqref{eq_scadiscomder} that 
	\begin{align*}
 		D_\alpha \left( \tau_{-k}^\alpha \rho_{-1}^\alpha w_{\bi,\bj} \right) =&~ \frac{1}{\ri\omega} \frac{ \rho_1^\alpha ( \tau_{-k}^\alpha \rho_{-1}^\alpha w_{\bi,\bj} ) - \tau_1^\alpha ( \tau_{-k}^\alpha \rho_{-1}^\alpha w_{\bi,\bj} ) }{ \rho_1^\alpha ( \tau_{-k}^\alpha \rho_{-1}^\alpha z_{\bi,\bj}^\alpha ) - \tau_1^\alpha ( \tau_{-k}^\alpha \rho_{-1}^\alpha z_{\bi,\bj}^\alpha ) } \\
 		=&~ \frac{1}{\ri\omega} \frac{ \tau_{-k}^\alpha w_{\bi,\bj} - \tau_{-k+1}^\alpha \rho_{-1}^\alpha w_{\bi,\bj} }{ \ri \frac{h}{\omega} \sigma_{j_\alpha-1}^\alpha } \\
 		=&~ \frac{ \tau_{-k}^\alpha w_{\bi,\bj} - \tau_{-k+1}^\alpha \rho_{-1}^\alpha w_{\bi,\bj} }{ - h \rho_{-1}^\alpha \sigma_{\bj}^\alpha },
	\end{align*}
	which implies by induction that
	\begin{align*}
 		\tau_{-k}^\alpha w_{\bi,\bj} =&~ \tau_{-k+1}^\alpha \rho_{-1}^\alpha w_{\bi,\bj} - h \rho_{-1}^\alpha \sigma_{\bj}^\alpha \cdot D_\alpha \left( \tau_{-k}^\alpha \rho_{-1}^\alpha w_{\bi,\bj} \right) \\
 		=&~ \tau_{-k}^\alpha \left( \tau_1^\alpha \rho_{-1}^\alpha w_{\bi,\bj} \right) - h \rho_{-1}^\alpha  \left( \sigma_{\bj}^\alpha \left( \tau_{-k}^\alpha D_\alpha w_{\bi,\bj} \right) \right) \\
 		=&~  \tau_{-k}^\alpha \left( \tau_2^\alpha \rho_{-2}^\alpha w_{\bi,\bj} \right) - h \rho_{-1}^\alpha  \left( \sigma_{\bj}^\alpha \left( \tau_{-k}^\alpha D_\alpha \left( \tau_1^\alpha \rho_{-1}^\alpha w_{\bi,\bj} \right) \right) \right)  - h \rho_{-1}^\alpha  \left( \sigma_{\bj}^\alpha \left( \tau_{-k}^\alpha D_\alpha w_{\bi,\bj} \right) \right) \\
 		=&~  \tau_{-k}^\alpha \left( \tau_2^\alpha \rho_{-2}^\alpha w_{\bi,\bj} \right) - h \rho_{-2}^\alpha  \left( \sigma_{\bj}^\alpha \left( \tau_{-k+1}^\alpha D_\alpha w_{\bi,\bj}  \right) \right)  - h \rho_{-1}^\alpha \left( \sigma_{\bj}^\alpha \left( \tau_{-k}^\alpha  D_\alpha w_{\bi,\bj} \right) \right)  \\
 		=&~ \cdots \\
 		=&~ \tau_0^\alpha \rho_{-k}^\alpha w_{\bi,\bj} - h \sum_{l=1}^k \rho_{-l}^\alpha  \left(  \sigma_{\bj}^\alpha\left( \tau_{-k+l-1}^\alpha D_\alpha  w_{\bi,\bj} \right) \right).
	\end{align*}
	This proves the identity \eqref{eq_taukw}. Note that in the above proof we have used the fact that $D_\alpha (\tau_r^\alpha \cdot) = \tau_r^\alpha (D_\alpha\cdot)$ and $D_\alpha (\rho_r^\alpha \cdot) = \rho_r^\alpha (D_\alpha\cdot)$ for $r\in\bbz$.
	
	Similar to the derivation of the identity~\eqref{eq_taukw}, \eqref{eq_taukDw} follows from the fact that $D_{\beta} w_{\bi,\bj}$ is discrete-holomorphic with respect to \( z^{\alpha} \) by Th.~\ref{th_da12}
	\begin{align*}
 		\tau_{-k}^{\alpha} D_{\beta} w_{\bi,\bj} = D_{\beta} ( \tau_{-k}^{\alpha} w_{\bi,\bj} )  =&~ D_\beta\left( \rho_{-k}^{\alpha} w_{\bi,\bj} - h \sum_{l=1}^k \rho_{-l}^{\alpha} \left(  \sigma_{\bj}^{\alpha}\left( \tau_{-k+l-1}^{\alpha}  D_{\alpha} w_{\bi,\bj} \right) \right) \right)\\
 		=&~  \rho_{-k}^{\alpha}  D_{\beta} w_{\bi,\bj}  - h \sum_{l=1}^k \rho_{-l}^{\alpha}  \left(  \sigma_{\bj}^{\alpha}\left( \tau_{-k+l-1}^{\alpha} D_{\alpha,\beta} w_{\bi,\bj} \right) \right).
	\end{align*}
	The proof is completed. 
\end{proof}

Using the above proposition, we further convert \( \tau_{k_1}^1 \tau_{k_2}^2 w_{\bi,\bj}  \) into \( \rho_{k_1}^1 \rho_{k_2}^2 w_{\bi,\bj} \) for all \( k_1,k_2\in\bbz\setminus\{0\} \).
\begin{proposition} \label{pro_tautauu}
 	Assume that $w_{\bi,\bj}\in\mathbb{C}^{\bbz^2\times\bbz^2}$ is discrete-holomorphic on \( (\Lambda, z^\alpha) \) for each \( \alpha \). For $k_1,k_2\in \bbz^+$, it holds 
 	\begin{align}
 		\tau_{-k_1}^1 \tau_{-k_2}^2 w_{\bi,\bj} =&~ \rho_{-k_1}^1 \rho_{-k_2}^2 w_{\bi,\bj} - h \sum_{\alpha\neq\beta} \sum_{l=1}^{k_\alpha}  \rho_{-l}^\alpha \rho_{-k_{\beta}}^{\beta} \left( \sigma_{\bj}^\alpha \left( \tau_{-k_\alpha+l-1}^\alpha D_\alpha w_{\bi,\bj} \right) \right) \notag\\
 		& + h^2 \sum_{l=1}^{k_1} \sum_{m=1}^{k_2}  \rho_{-l}^1 \rho_{-m}^2 \left( \sigma_{\bj}^1 \sigma_{\bj}^2 \left( \tau_{-k_1+l-1}^1 \tau_{-k_2+m-1}^2 D_{1,2} w_{\bi,\bj} \right) \right) \label{eq_tmkmk1}\\
 		\tau_{k_1}^1 \tau_{k_2}^2 w_{\bi,\bj} =&~ \rho_{k_1}^1 \rho_{k_2}^2 w_{\bi,\bj} + h \sum_{\alpha\neq\beta} \sum_{l=1}^{k_\alpha}  \rho_{l-1}^\alpha \rho_{k_{\beta}}^{\beta} \left( \sigma_{\bj}^\alpha \left( \tau_{k_\alpha-l}^\alpha D_\alpha w_{\bi,\bj} \right) \right) \notag\\
 		& + h^2 \sum_{l=1}^{k_1} \sum_{m=1}^{k_2}  \rho_{l-1}^1 \rho_{m-1}^2 \left( \sigma_{\bj}^1 \sigma_{\bj}^2 \left( \tau_{k_1-l}^1 \tau_{k_2-m}^2 D_{1,2} w_{\bi,\bj} \right) \right). \label{eq_tmkmk2}
	\end{align}
	When $\alpha\neq\beta$, and $k_{\alpha},k_{\beta}\in \bbz^+$,
	\begin{align}
 		\tau_{-k_{\alpha}}^{\alpha} \tau_{k_{\beta}}^{\beta} w_{\bi,\bj} =&~ \rho_{-k_{\alpha}}^{\alpha} \rho_{k_{\beta}}^{\beta} w_{\bi,\bj} - h \sum_{l=1}^{k_{\alpha}} \rho_{-l}^{\alpha} \rho_{k_{\beta}}^{\beta} \left( \sigma_{\bj}^{\alpha} \left( \tau_{-k_{\alpha}+l-1}^{\alpha} D_{\alpha} w_{\bi,\bj} \right) \right) \notag\\
 		& + h \sum_{l=1}^{k_{\beta}} \rho_{l-1}^{\beta} \rho_{-k_{\alpha}}^{\alpha} \left( \sigma_{\bj}^{\beta} \left( \tau_{k_{\beta}-l}^{\beta} D_{\beta} w_{\bi,\bj} \right) \right) \notag\\
 		& - h^2 \sum_{l=1}^{k_{\alpha}} \sum_{m=1}^{k_{\beta}} \rho_{-l}^\alpha \rho_{m-1}^\beta \left( \sigma_{\bj}^\alpha \sigma_{\bj}^\beta \left( \tau_{-k_\alpha+l-1}^\alpha \tau_{k_\beta-m}^\beta D_{1,2} w_{\bi,\bj} \right) \right). \label{eq_tmkmk3}
	\end{align}
\end{proposition}

\begin{proof}
	Here we only prove the identity \eqref{eq_tmkmk1}. The proofs for other identities are analogous. It is clear that $\tau_{-k_2}^2 w_{\bi,\bj}$ is discrete-holomorphic with respect to $(i_1,j_1)$. By \eqref{eq_taukw} and \eqref{eq_taukDw}, it has 
	\begin{align*}
 		\tau_{-k_1}^1 \tau_{-k_2}^2 w_{\bi,\bj} =&~ \tau_{-k_1}^1 \left( \rho_{-k_2}^2 w_{\bi,\bj} - h \sum_{l=1}^{k_2} \rho_{-l}^2  \left(  \sigma_{\bj}^2 \left(	\tau_{-k_2+l-1}^2 D_2  w_{\bi,\bj} \right) \right) \right) \\
 		=&~ \tau_{-k_1}^1 \rho_{-k_2}^2 w_{\bi,\bj}  - h \sum_{l=1}^{k_2} \rho_{-l}^2  \left(  \sigma_{\bj}^2 \tau_{-k_2+l-1}^2 \left( \tau_{-k_1}^1   D_2  w_{\bi,\bj} \right) \right) \notag\\
 		=&~ \rho_{-k_2}^2 \left( \rho_{-k_1}^1 w_{\bi,\bj} - h \sum_{l=1}^{k_1} \rho_{-l}^1 \left( \sigma_{\bj}^1 \left( \tau_{-k_1+l-1}^1 D_1 w_{\bi,\bj} \right) \right) \right) \notag\\
 		& - h \sum_{l=1}^{k_2} \rho_{-l}^2  \left(  \sigma_{\bj}^2 \tau_{-k_2+l-1}^2 \left( \rho_{-k_1}^1 D_2 w_{\bi,\bj} - h \sum_{m=1}^{k_1} \rho_{-m}^1 \left( \sigma_{\bj}^1 \tau_{-k_1+m-1} D_{1,2} w_{\bi,\bj}  \right)  \right) \right) \notag\\
 		=&~ \rho_{-k_1}^1 \rho_{-k_2}^2 w_{\bi,\bj} - h \sum_{l=1}^{k_1} \rho_{-l}^1 \rho_{-k_2}^2 \left( \sigma_{\bj}^1 \left( \tau_{-k_1+l-1}^1 D_1 w_{\bi,\bj} \right) \right) \notag\\
 		& - h \sum_{l=1}^{k_2} \rho_{-k_1}^1 \rho_{-l}^2 \left( \sigma_{\bj}^2 \tau_{-k_2+l-1}^2 D_2 w_{\bi,\bj} \right) \notag\\
 		& + h^2 \sum_{m=1}^{k_1} \sum_{l=1}^{k_2} \rho_{-l}^2 \rho_{-m}^1 \left( \sigma_{\bj}^1 \sigma_{\bj}^2   \tau_{-k_2+l-1}^2   \tau_{-k_1+m-1} D_{1,2} w_{\bi,\bj}  \right).
	\end{align*}
	The proof is completed. 
\end{proof}

The following proposition shows the identities of the shifted discrete derivatives are satisfied.
\begin{proposition} \label{pro_tildepsi}
 	Assume that $w_{\bi,\bj}\in\mathbb{C}^{\bbz^2\times\bbz^2}$ is discrete-holomorphic on \( (\Lambda, z^\alpha) \) for each \( \alpha \). For \( k\in \mathbb{N} \), it holds 
 	\begin{align}
 		\ri \omega \tau_k^\alpha D_\alpha  w_{\bi,\bj} =&~ \frac{ (\rho_{k+2}^\alpha - \rho_k^\alpha) w_{\bi,\bj} }{2h} + \frac{ (\rho_1^\alpha+\rho_0^\alpha) (\sigma_{\bj}^\alpha \tau_{k}^\alpha D_\alpha w_{\bi,\bj}) }{2}  + \sum_{l=1}^k \frac{ (\rho_{l+1}^\alpha - \rho_{l-1}^\alpha) ( \sigma_{\bj}^\alpha \tau_{k-l}^\alpha D_\alpha w_{\bi,\bj} ) }{2} \label{eq_tdw1}
 	\end{align}
 	and
 	\begin{align}
 		\ri \omega \tau_{-k}^\alpha D_\alpha  w_{\bi,\bj} =&~ \frac{ (\rho_{-k+2}^\alpha - \rho_{-k}^\alpha) w_{\bi,\bj} }{2h} + \frac{ (\rho_{-1}^\alpha+\rho_0^\alpha) (\sigma_{\bj}^\alpha \tau_{-k}^\alpha D_\alpha w_{\bi,\bj}) }{2} \notag\\
		&~+ \sum_{l=1}^k \frac{ (\rho_{-l-1}^\alpha - \rho_{-l+1}^\alpha) ( \sigma_{\bj}^\alpha \tau_{-k+l}^\alpha D_\alpha w_{\bi,\bj} ) }{2}.\label{eq_tdw2}
	\end{align}
\end{proposition}

\begin{proof}
	By the definitions of \( D_\alpha \) and \( z_{i,j}^\alpha \), i.e., Eq.~\eqref{eq_scadiscomder} and Eq.~\eqref{eq_defzij}, we have
	\begin{align}
 		\tau_k^\alpha D_\alpha w_{\bi,\bj} = D_\alpha (\tau_k^\alpha w_{\bi,\bj}) = \frac{1}{\ri \omega} \frac{ \rho_1^\alpha \tau_1^\alpha (\tau_k^\alpha w_{\bi,\bj}) - (\tau_k^\alpha w_{\bi,\bj}) }{ \rho_1^\alpha \tau_1^\alpha z_{\bi,\bj}^\alpha -  z_{\bi,\bj}^\alpha } = \frac{1}{\ri \omega} \frac{ \rho_1^\alpha \tau_{k+1}^\alpha w_{\bi,\bj} - \tau_k^\alpha w_{\bi,\bj} }{ 2h + \ri \frac{h}{\omega} \sigma_{\bj}^\alpha }
	\end{align}
	which, combining with Eq.~\eqref{eq_taukw2}, implies
	\begin{align}
 		& (2\ri h \omega - h \sigma_{\bj}^\alpha ) \tau_k^\alpha D_\alpha w_{\bi,\bj} \notag\\
 		=&~ \tau_{k+1}^\alpha (\rho_1^\alpha w_{\bi,\bj}) - \tau_k^\alpha w_{\bi,\bj} \notag\\
 		=&~ \left( \rho_{k+1}^\alpha (\rho_1^\alpha w_{\bi,\bj}) + h \sum_{l=1}^{k+1} \rho_{l-1}^\alpha \left( \sigma_{j_\alpha+1}^\alpha \left( \tau_{k+1-l}^\alpha D_\alpha (\rho_1^\alpha w_{\bi,\bj}) \right) \right) \right) \notag\\
 		& - \left( \rho_k^\alpha w_{\bi,\bj} + h \sum_{l=1}^k \rho_{l-1}^\alpha \left( \sigma_{\bj}^\alpha \left( \tau_{k-l}^\alpha D_\alpha w_{\bi,\bj} \right) \right) \right) \notag \\
 		=&~ \rho_{k+2}^\alpha w_{\bi,\bj} + h \sum_{l=1}^{k+1} \rho_{l}^\alpha \left( \sigma_{\bj}^\alpha \left( \tau_{k+1-l}^\alpha D_\alpha w_{\bi,\bj} \right) \right) - \rho_k^\alpha w_{\bi,\bj} -  h \sum_{l=1}^k \rho_{l-1}^\alpha \left( \sigma_{\bj}^\alpha \left( \tau_{k-l}^\alpha D_\alpha w_{\bi,\bj} \right) \right) \notag\\
 		=&~ \left( \rho_{k+2}^\alpha - \rho_{k}^\alpha \right) w_{\bi,\bj} + h \rho_{1}^\alpha \left( \sigma_{\bj}^\alpha \left( \tau_{k}^\alpha D_\alpha w_{\bi,\bj} \right) \right) + h \sum_{l=1}^k \left( \rho_{l+1}^\alpha - \rho_{l-1}^\alpha \right) \left( \sigma_{\bj}^\alpha \left( \tau_{k-l}^\alpha D_\alpha w_{\bi,\bj} \right) \right).
	\end{align}
	This proves the identity \eqref{eq_tdw1}. The proof of \eqref{eq_tdw2} is analogous.
\end{proof}

\subsection{Discrete PML equations}
Here we will give the discrete PML equations and the derivation is postponed in Section~\ref{subsec_dodpe}.

Using a discrete version of complex coordinate stretching, the discrete PML for the discrete wave equation~\eqref{eq_diswaveeq} are derived: for $\bi$ in the PML domain, it has 
\begin{align}
	\partial_t^2 u_\bi =&~ \mathcal{L}_h u_\bi + h \sum_{\alpha=1}^2 \sum_{k\in\mathbb{Z}^+} \sum_{l=0}^{k-1} a_{k} \left( \tau_{l}^{\alpha} \left( \sigma_{\bi}^\alpha \tilde{\psi}_{\bi}^{\alpha,k-l} \right) -  \tau_{-l-1}^{\alpha} \left( \sigma_{\bi}^\alpha \bar{\psi}_\bi^{\alpha,k-l}\right) \right)  \notag\\
	& + h \sum_{\alpha\neq\beta} \sum_{\bk \in(\mathbb{Z}^+)^2} \sum_{l=0}^{k_\alpha-1} a_{\bk} \left(  \tau_{-k_{\beta}}^{\beta}  + \tau_{k_{\beta}}^{\beta}  \right)  \left(  \tau_{l}^{\alpha}  \left(  \sigma_{\bi}^\alpha \tilde{\psi}_\bi^{\alpha,k_\alpha-l} \right) -   \tau_{-l-1}^{\alpha} \left( \sigma_{\bi}^\alpha \bar{\psi}_\bi^{\alpha,k_\alpha-l} \right)  \right) \notag\\
	& + h^2 \sum_{\bk\in(\mathbb{Z}^+)^2} \sum_{\substack{0\leq l\leq k_1-1\\ 0\leq m\leq k_2-1}} a_{\bk}  \Bigg[ \tau_{l}^{1} \tau_{m}^{2} \left( \sigma_{\bi}^{1} \sigma_{\bi}^{2}  \tilde{\tilde{\psi}}_\bi^{k_1-l,k_2-m} \right)  - \tau_{-l-1}^{1} \tau_{m}^{2} \left( \sigma_{\bi}^{1} \sigma_{\bi}^{2} \tilde{\bar{\psi}}_\bi^{k_1-l,k_2-m}  \right) \notag\\
	& - \tau_{l}^{1} \tau_{-m-1	}^{2} \left( \sigma_{\bi}^{1} \sigma_{\bi}^{2} \bar{\tilde{\psi}}_\bi^{k_1-l,k_2-m} \right) + \tau_{-l-1}^{1} \tau_{-m-1}^{2} \left(  \sigma_{\bi}^{2} \sigma_{\bi}^{2} \bar{\bar{\psi}}_\bi^{k_1-l,k_2-m} \right) \Bigg], \label{eq_disPMLmaineq}
\end{align}
where $a_k=a_{0,k}=a_{k,0}$, the auxiliary variables $\tilde{\psi}_\bi^{\alpha,k}$ and $\bar{\psi}_\bi^{\alpha,k}$ satisfy the following equations
\begin{align}
	\partial_t \tilde{\psi}_\bi^{\alpha,k} =&~ \frac{ ( \tau_{k-1}^{\alpha} - \tau_{k+1}^{\alpha}  ) u_\bi }{2h} - \frac{ ( \tau_{1}^{\alpha} + \tau_0^\alpha ) ( \sigma_{\bi}^\alpha \tilde{\psi}_\bi^{\alpha,k} )  }{2}  -  \sum_{l=1}^{k-1} \frac{ (  \tau_{l+1}^{\alpha} - \tau_{l-1}^{\alpha} ) ( \sigma_{\bi}^\alpha \tilde{\psi}_\bi^{\alpha,k-l} )  }{2}, \label{eq_disPMLmaineq1}\\
	\partial_t \bar{\psi}_\bi^{\alpha,k} =&~ \frac{ ( \tau_{-k}^{\alpha} - \tau_{-k+2}^{\alpha} ) u_\bi }{2h} -  \frac{ ( \tau_0^\alpha + \tau_{-1}^{\alpha} ) ( \sigma_{\bi}^\alpha \bar{\psi}_\bi^{\alpha,k} ) }{2}  -  \sum_{l=1}^{k-1} \frac{ ( \tau_{-l-1}^{\alpha}  - \tau_{-l+1}^{\alpha} ) ( \sigma_{\bi}^\alpha \bar{\psi}_\bi^{\alpha,k-l} ) }{2}, \label{eq_disPMLmaineq2}
\end{align}
and auxiliary variables \( \tilde{\tilde{\psi}}_\bi^\bk, \tilde{\bar{\psi}}_\bi^\bk, \bar{\tilde{\psi}}_\bi^\bk \) and \( \bar{\bar{\psi}}_\bi^\bk \) satisfy the following equations
\begin{align}
	\partial_t  \tilde{\tilde{\psi}}_\bi^\bk =&~ \frac{ ( \tau_{k_1-1}^{1} - \tau_{k_1+1}^{1}  ) \tilde{\psi}_\bi^{2,k_2} }{2h} - \frac{ ( \tau_{1}^{1} + \tau_0^1 ) ( \sigma_{\bi}^1 \tilde{\tilde{\psi}}_\bi^\bk )  }{2}  -  \sum_{l=1}^{k_1-1} \frac{ ( \tau_{l+1}^{1} - \tau_{l-1}^{1} ) ( \sigma_{\bi}^1 \tilde{\tilde{\psi}}_\bi^{k_1-l,k_2} ) }{2}, \label{eq_disPMLmaineq3}\\
	\partial_t \tilde{\bar{\psi}}_\bi^\bk =&~ \frac{ ( \tau_{-k_1}^{1} - \tau_{-k_1+2}^{1} ) \tilde{\psi}_\bi^{2,k_2}  }{2h} -  \frac{ ( \tau_0^1 + \tau_{-1}^{1} ) ( \sigma_{\bi}^1  \tilde{\bar{\psi}}_\bi^\bk ) }{2}  -  \sum_{l=1}^{k_1-1} \frac{ ( \tau_{-l-1}^{1} - \tau_{-l+1}^{1} ) ( \sigma_{\bi}^1  \tilde{\bar{\psi}}_\bi^{k_1-l,k_2} ) }{2},  \label{eq_disPMLmaineq4}\\
	\partial_t \bar{\tilde{\psi}}_\bi^\bk =&~ \frac{ (  \tau_{k_1-1}^{1}  - \tau_{k_1+1}^{1} ) \bar{\psi}_\bi^{2,k_2} }{2h} - \frac{ ( \tau_{1}^{1} + \tau_0^1 ) ( \sigma_{\bi}^1 \bar{\tilde{\psi}}_\bi^\bk )  }{2}  -  \sum_{l=1}^{k_1-1} \frac{ ( \tau_{l+1}^{1} - \tau_{l-1}^{1} ) ( \sigma_{\bi}^1 \bar{\tilde{\psi}}_\bi^{k_1-l,k_2} ) }{2}, \label{eq_disPMLmaineq5}\\
	\partial_t \bar{\bar{\psi}}_\bi^\bk =&~ \frac{ ( \tau_{-k_1}^{1} - \tau_{-k_1+2}^{1}  ) \bar{\psi}_\bi^{2,k_2}  }{2h} - \frac{ ( \tau_0^1 + \tau_{-1}^{1} ) ( \sigma_{\bi}^1  \bar{\bar{\psi}}_\bi^\bk ) }{2}  -  \sum_{l=1}^{k_1-1} \frac{ ( \tau_{-l-1}^{1} - \tau_{-l+1}^{1} ) ( \sigma_{\bi}^1  \bar{\bar{\psi}}_\bi^{k_1-l,k_2} ) }{2}. \label{eq_disPMLmaineq6}
\end{align}
$\sigma^\alpha_\bi=\sigma_{i_\alpha}^\alpha\geq0$ are the PML damping coefficients, with $\sigma^\alpha_{i_\alpha}=0$ for \( i_\alpha<0 \). In the physical domain where $\sigma^\alpha=0$ for all \( \alpha=1,2 \), the auxiliary variables are set to be zero and Eq.~\eqref{eq_disPMLmaineq} is reduced to the discrete wave equation.~\eqref{eq_diswaveeq}. The auxiliary variables \( \tilde{\psi}^\alpha_\bi,\bar{\psi}^\alpha_\bi \) and their evolution equations~\eqref{eq_disPMLmaineq1}--\eqref{eq_disPMLmaineq2} only need to be supported in the half space \( \{i_\alpha\geq 0\} \), that is, the discrete wave equation is a special case of the discrete PML equations with zero damping. Similarly, the auxiliary variables \( \tilde{\tilde{\psi}}_\bi, \tilde{\bar{\psi}}_\bi, \bar{\tilde{\psi}}_\bi, \bar{\bar{\psi}}_\bi \) and their evolution equations~\eqref{eq_disPMLmaineq3}--\eqref{eq_disPMLmaineq6} only need to be supported in the corner \( \{i_1\geq 0\}\cap\{i_2\geq 0\} \).

The derivation of PML equations \eqref{eq_disPMLmaineq}--\eqref{eq_disPMLmaineq6} is based on the idea discussed in the beginning of this section (see \eqref{eq_disanacondiswave}--\eqref{est_exp}), namely, we expect the fundamental solution (i.e., the holomorphic extension of $e^{\ri \kappa_\alpha^h \cdot \bi h}$ defined as $u_{\bi,\bj}$ in \eqref{eq_disanacondiswave2}) to the resulting PML equations to decay exponentially in PML domain, which is the following theorem.

The feasible values of numerical wave numbers \( \kappa_\alpha^h \) are further restricted to
\begin{align}
 	\mathcal{K} := \{ \kappa_\alpha^h\subset\mathbb{C}|\ -\frac{\pi}{h} < \Re\kappa_\alpha^h \leq \frac{\pi}{h},\ \Im \kappa_\alpha^h <0 \}, \quad \alpha=1,2,
\end{align}
using the periodicity of \( e^{\ri \kappa_\alpha^h i_\alpha h} \) and the boundedness assumption of \( u_\bi(t) \). For waves propagating in the \( +x_\alpha \) direction, \( \mathrm{sgn}(\omega)\Re \kappa_\alpha^h \) is positive. The wave that has \( \Im\kappa_\alpha^h=0 \) for all \( \alpha \) is called a plane wave. The wave that has \( \Im \kappa_\alpha <0 \) for some \( \alpha \) is called an evanescent wave (cf. \cite{chern2019reflectionless}).

\begin{lemma}	\label{th_main1}
 {For \( \kappa_\alpha^h \in \mathcal{K} \), \( \omega\in\mathbb{R} \setminus\{0\} \) and \( \alpha \in \{1,2\} \) , the holomorphic extension  \( e^{-\ri\omega t} u_{\bi,\bj} \) of the fundamental solution \( e^{-\ri\omega t} e^{\ri \bka^h\cdot\bi h} \) to Eq. \eqref{eq_diswaveeq} is discrete holomorphic on \( (\Lambda, z^\alpha) \) and still satisfies Eq.~\eqref{eq_diswaveeq}, where}
 	\begin{align}
 	 	u_{\bi,\bj} = \prod_{\alpha=1}^2 \eta(j_\alpha,\sigma^\alpha,\kappa_\alpha^h,\omega) e^{\ri \kappa_\alpha^h (i_\alpha+j_\alpha) h}\label{eq_disanacondiswave2}
	\end{align}
	with 
	\begin{equation*} \eta(j_\alpha,\sigma^\alpha,\kappa_\alpha^h,\omega)=  \prod_{l=0}^{j_\alpha-1} \frac{ 2+\ri\frac{\sigma^\alpha_l}{\omega}(1-e^{-\ri \kappa_\alpha^h h}) }{ 2+\ri\frac{\sigma^\alpha_l}{\omega}(1-e^{\ri \kappa_\alpha^h h}) }.
	\end{equation*}
	Moreover, for \( \omega\in\mathbb{R}^+, \sigma^\alpha>0, \kappa_\alpha^h\in\mathcal{K} \) and \( \mathrm{sgn}(\omega)\Re(\kappa_\alpha^h)>0 \), it holds that
	\begin{align}
 		\left| \frac{ 2+\ri\frac{\sigma^\alpha_l}{\omega}(1-e^{-\ri \kappa_\alpha^h h}) }{ 2+\ri\frac{\sigma^\alpha_l}{\omega}(1-e^{\ri \kappa_\alpha^h h}) } \right| < 1. \label{eq_estdampcoe}
	\end{align}
\end{lemma}
\begin{proof}
	Since \(  e^{-\ri\omega t} e^{\ri \bka^h\bi h} \) is the solution to \eqref{eq_diswaveeq} and \( \tau^\alpha \) is the translation operator with respect to \( i_\alpha \), it is easy to get that \( u_{\bi,\bj} \) is also the solution to \eqref{eq_diswaveeq}. The proof for the discrete-holomorphicity of \( u_{\bi,\bj} \) and the estimate \eqref{eq_estdampcoe} follows from the \cite[Theorem 3.1]{chern2019reflectionless}.
\end{proof}

\subsection{Derivation of discrete PML equations} \label{subsec_dodpe}
In this section, we derive the PML equations~\eqref{eq_disPMLmaineq}--\eqref{eq_disPMLmaineq6} for the discrete wave equation~\eqref{eq_diswaveeq} on the PML domain. We first complexify the discrete wave equation~\eqref{eq_diswaveeq} into the PML equations~\eqref{eq_disPMLmaineq} that \( \prod_{\alpha=1}^2 \eta(i_\alpha,\sigma^\alpha,\kappa_\alpha^h,\omega) e^{\ri \kappa_\alpha^h i_\alpha h} \) satisfies.

It can be observed that \( a_{k,0}=a_{0,k} \) for any \( k\in\bbz \) and \( a_\bk = a_{-k_1,k_2} = a_{k_1,-k_2} = a_{-k_1,-k_2} \) for any \( \bk\in\bbz^2 \). To derive the discrete PML, we denote by \( a_k=a_{0,k}=a_{k,0} \), and rewrite the discrete nonlocal operator~\eqref{eq_deflhui} as
\begin{align}
	\mathcal{L}_h u_\bi(t) =&~ a_0 u_\bi(t) + \sum_{\alpha=1}^2 \sum_{k\in\mathbb{Z}^+} a_{k}  ( \tau_{k}^{\alpha} + \tau_{-k}^{\alpha} ) u_\bi(t)  \notag\\
	&+ \sum_{\bk=(k_1,k_2)\in(\mathbb{Z}^+)^2} a_\bk ( \tau_{k_1}^{1} \tau_{k_2}^{2} + \tau_{-k_1}^{1} \tau_{k_2}^{2} + \tau_{k_1}^{1} \tau_{-k_2}^{2} + \tau_{-k_1}^{1} \tau_{-k_2}^{2} )  u_{\bi}(t),
\end{align}
where $\tau_k^1$ and $\tau_k^2$ are the translation operators $\tau_k^1 u_\bi = u_{i_1+k,i_2}, \tau_k^2 u_\bi = u_{i_1,i_2+k}$.

We begin with the application of the Fourier transform in \( t \) to \eqref{eq_diswaveeq}, namely 
\begin{align}
 	-\omega^2 u_\bi =&~ a_0 u_\bi + \sum_{\alpha=1}^2 \sum_{k\in\mathbb{Z}^+} a_{k}  ( \tau_{k}^{\alpha} + \tau_{-k}^{\alpha} ) u_\bi  \notag\\
	&+ \sum_{\bk=(k_1,k_2)\in(\mathbb{Z}^+)^2} a_\bk ( \tau_{k_1}^{1} \tau_{k_2}^{2} + \tau_{-k_1}^{1} \tau_{k_2}^{2} + \tau_{k_1}^{1} \tau_{-k_2}^{2} + \tau_{-k_1}^{1} \tau_{-k_2}^{2} )  u_{\bi},\quad \bi\in \Omega_p^+, \label{eq_disfou2}
\end{align}
where \( u_\bi = u_\bi(\omega) \) is the Fourier transform of \( u_\bi(t) \). Its fundamental solution is \( u_\bi = e^{\ri \bka^h\cdot\bi h} \) with \( \bka^h \) satisfying the discrete dispersion relation \eqref{eq_disdeprel}.

We first introduce some notations for ease of presentation as 
\begin{align}
 	&\tilde{\psi}_\bj^{\alpha,k+1} = \tau_{k}^\alpha D_\alpha u_{\b0,\bj},\qquad\qquad \bar{\psi}_\bj^{\alpha,k} = \tau_{-k}^\alpha D_\alpha u_{\b0,\bj},\qquad\qquad\quad \alpha=1,2,k\in\bbz^+, \\
 	&\tilde{\tilde{\psi}}_\bj^{k_1+1,k_2+1} = \tau_{k_1}^1 \tau_{k_2}^2 D_{1,2} u_{\b0,\bj},\ \bar{\tilde{\psi}}_\bj^{k_1+1,k_2} = \tau_{k_1}^1 \tau_{-k_2}^2 D_{1,2} u_{\b0,\bj},\quad k_1,k_2\in\bbz^+, \\
 	& \tilde{\bar{\psi}}_\bj^{k_1,k_2+1} = \tau_{-k_1}^1 \tau_{k_2}^2 D_{1,2} u_{\b0,\bj},\quad \bar{\bar{\psi}}_\bj^{k_1,k_2} = \tau_{-k_1}^1 \tau_{-k_2}^2 D_{1,2} u_{\b0,\bj},\quad k_1,k_2\in\bbz^+.
\end{align}
Substituting \( u_{\bi,\bj} \) given in Lemma~\ref{th_main1} into \eqref{eq_disfou2} yields
\begin{align}
 	-\omega^2 u_{\b0,\bj} =&~ a_0 u_{\b0,\bj} + \sum_{\alpha=1}^2 \sum_{k\in\mathbb{Z}^+} a_{k} ( \tau_{k}^{\alpha} + \tau_{-k}^{\alpha} ) u_{\b0,\bj}  \notag\\
 	&+ \sum_{\bk\in(\mathbb{Z}^+)^2} a_\bk ( \tau_{k_1}^{1} \tau_{k_2}^{2} + \tau_{-k_1}^{1} \tau_{k_2}^{2} + \tau_{k_1}^{1} \tau_{-k_2}^{2} + \tau_{-k_1}^{1} \tau_{-k_2}^{2} )  u_{\b0,\bj}. \label{eq_dermaineq1}
\end{align}
Since $u_{\bi,\bj}$ is discrete-holomorphic on \( (\Lambda, z^\alpha) \) for each \( \alpha \), it follows from Proposition~\ref{pro_tauw} that 
\begin{align}
 	&~ ( \tau_{k}^{\alpha} + \tau_{-k}^{\alpha} ) u_{\b0,\bj} \notag \\
	=&~ \rho_k^\alpha u_{\b0,\bj} + h\sum_{l=1}^{k} \rho_{l-1}^\alpha \left( \sigma_{\bj}^\alpha \left( \tau_{k-l}^\alpha D_\alpha u_{\b0,\bj} \right) \right) + \rho_{-k}^\alpha u_{\b0,\bj} + h\sum_{l=1}^{k} \rho_{-l}^\alpha \left( \sigma_{\bj}^\alpha \left( \tau_{-k+l-1}^\alpha D_\alpha u_{\b0,\bj} \right) \right) \notag\\
 	=&~ ( \rho_k^\alpha + \rho_{-k}^\alpha ) u_{\b0,\bj} + h \sum_{l=1}^{k} \left( \rho_{l-1}^\alpha \left( \sigma_{\bj}^\alpha \tilde{\psi}_\bj^{\alpha,k-l+1} \right) - \rho_{-l}^\alpha \left( \sigma_{\bj}^\alpha \bar{\psi}_\bj^{\alpha,k-l+1} \right)  \right). \label{eq_dermaineq2}
\end{align}
Similarly, from Proposition~\ref{pro_tautauu},  it has 
\begin{align}
 	&~ \left( \tau_{k_1}^{1} \tau_{k_2}^{2} + \tau_{-k_1}^{1} \tau_{k_2}^{2} + \tau_{k_1}^{1} \tau_{-k_2}^{2} + \tau_{-k_1}^{1} \tau_{-k_2}^{2} \right)  u_{\b0}^\bj \notag\\
 	=&~ \left( \rho_{k_1}^1 \rho_{k_2}^2 + \rho_{-k_1}^1 \rho_{k_2}^2 + \rho_{k_1}^1 \rho_{-k_2}^2 + \rho_{-k_1}^1 \rho_{-k_2}^2 \right) u_{\b0,\bj}  \notag\\
 	&~ +  h \sum_{\alpha\neq\beta} \sum_{l=1}^{k_\alpha} \left( \rho_{-k_{\beta}}^{\beta} + \rho_{k_{\beta}}^{\beta} \right) \left( \rho_{l-1}^\alpha  \left( \sigma_{\bj}^\alpha \tilde{\psi}_\bj^{\alpha,k_\alpha-l+1} \right) -  \rho_{-l}^\alpha  \left( \sigma_{\bj}^\alpha \bar{\psi}_\bj^{\alpha,k_\alpha-l+1} \right) \right) \notag\\
 	&~ + h^2 \sum_{\substack{1\leq l\leq k_1\\ 1\leq m\leq k_2}}  \Bigg[ \rho_{l-1}^1 \rho_{m-1}^2 \left( \sigma_{\bj}^1 \sigma_{\bj}^2 \tilde{\tilde{\psi}}_\bj^{k_1-l+1,k_2-m+1} \right) - \rho_{-l}^1 \rho_{m-1}^2 \left( \sigma_{\bj}^1 \sigma_{\bj}^2 \tilde{\bar{\psi}}_\bj^{k_1-l+1,k_2-m+1}  \right) \notag\\
 	&~ - \rho_{l-l}^2 \rho_{-m}^1 \left( \sigma_{\bj}^1 \sigma_{\bj}^2 \bar{\tilde{\psi}}_\bj^{k_1-l+1,k_2-m+1}  \right) + \rho_{-l}^1 \rho_{-m}^2 \left( \sigma_{\bj}^1 \sigma_{\bj}^2 \bar{\bar{\psi}}_\bj^{k_1-l+1,k_2-m+1}\right) \Bigg]. \label{eq_dermaineq3}
\end{align}
Combining \eqref{eq_dermaineq1}--\eqref{eq_dermaineq3}, we obtain
\begin{align}
 	&~ -\omega^2 u_{\b0,\bj} \notag\\
	=&~ \sum_{\bk\in\bbz^2} a_\bk u_{\b0,\bj+\bk} + \sum_{\alpha=1}^2 \sum_{k\in\mathbb{Z}^+}\sum_{l=1}^{k}  a_{k} \left( \rho_{l-1}^\alpha \left( \sigma_{\bj}^\alpha \tilde{\psi}_\bj^{\alpha,k-l+1} \right) - \rho_{-l}^\alpha \left( \sigma_{\bj}^\alpha \bar{\psi}_\bj^{\alpha,k-l+1} \right)  \right) \notag\\
 	&~ + h \sum_{\alpha\neq\beta}  \sum_{\bk\in(\mathbb{Z}^+)^2} \sum_{l=1}^{k_\alpha} a_\bk \left( \rho_{-k_{\beta}}^{\beta} + \rho_{k_{\beta}}^{\beta} \right) \left( \rho_{l-1}^\alpha  \left( \sigma_{\bj}^\alpha \tilde{\psi}_\bj^{\alpha,k_\alpha-l+1} \right) -  \rho_{-l}^\alpha  \left( \sigma_{\bj}^\alpha \bar{\psi}_\bj^{\alpha,k_\alpha-l+1} \right) \right) \notag\\
 	&~ + h^2  \sum_{\bk\in(\mathbb{Z}^+)^2} \sum_{\substack{1\leq l\leq k_1\\ 1\leq m\leq k_2}} a_\bk \Bigg[ \rho_{l-1}^1 \rho_{m-1}^2 \left( \sigma_{\bj}^1 \sigma_{\bj}^2 \tilde{\tilde{\psi}}_\bj^{k_1-l+1,k_2-m+1} \right) - \rho_{-l}^1 \rho_{m-1}^2 \left( \sigma_{\bj}^1 \sigma_{\bj}^2 \tilde{\bar{\psi}}_\bj^{k_1-l+1,k_2-m+1}  \right) \notag\\
 	&~ - \rho_{l-l}^2 \rho_{-m}^1 \left( \sigma_{\bj}^1 \sigma_{\bj}^2 \bar{\tilde{\psi}}_\bj^{k_1-l+1,k_2-m+1}  \right) + \rho_{-l}^1 \rho_{-m}^2 \left( \sigma_{\bj}^1 \sigma_{\bj}^2 \bar{\bar{\psi}}_\bj^{k_1-l+1,k_2-m+1}\right) \Bigg]. \label{eq_dermaineqfin}
\end{align}
This equation can be rewritten as the PML equation~\eqref{eq_disPMLmaineq} by replacing \( \rho, u_{\b0,\bj},\bj \) by \( \tau, u_\bi,\bi \), respectively, and transforming \( u_\bi \) and other auxiliary variables back to the time domain. It is clear that the discrete analytic continuation of \( u_\bi(t) = e^{-\ri \omega t} e^{\ri\boldsymbol{\kappa}^h\cdot \bi h} \), \( \bi\in(\bbz^+)^2 \), given by
\begin{align}
 	\tilde{u}_\bi = e^{-\ri \omega t} \prod_{\alpha=1}^2  \eta(i_\alpha,\sigma^\alpha,\kappa_\alpha^h,\omega)   e^{\ri \kappa_\alpha^h i_\alpha h},
\end{align}
is the solution to Eq.~\eqref{eq_disPMLmaineq}. It is clear that \( \tilde{u}_\bi=u_\bi \) in the physical domain  because of \( \sigma^\alpha=0 \), and \( \tilde{u}_\bi \) is decaying exponentially when \( \bi \) is in the PML domain.

Next we derive the PML equations~\eqref{eq_disPMLmaineq1}--\eqref{eq_disPMLmaineq2} for auxiliary variables $\tilde{\psi}^{\alpha,k},\bar{\psi}^{\alpha,k}$. From the identity~\eqref{eq_tdw1} in Proposition~\ref{pro_tildepsi}, by setting \( \bi=\b0 \),
\begin{align}
 	\ri \omega \tau_k^\alpha D_\alpha  u_{\b0,\bj} =&~ \frac{ (\rho_{k+2}^\alpha - \rho_k^\alpha) u_{\b0,\bj} }{2h} + \frac{ (\rho_1^\alpha+\rho_0^\alpha) (\sigma_{\bj}^\alpha \tau_{k}^\alpha D_\alpha u_{\b0,\bj}) }{2}  + \sum_{l=1}^k \frac{ (\rho_{l+1}^\alpha - \rho_{l-1}^\alpha) ( \sigma_{\bj}^\alpha \tau_{k-l}^\alpha D_\alpha u_{\b0,\bj} ) }{2}
\end{align}
which is equivalent to
\begin{align}
 	\ri \omega \tilde{\psi}_\bj^{\alpha,k}  =&~ \frac{ (\rho_{k+1}^\alpha - \rho_{k-1}^\alpha) u_{\b0,\bj} }{2h} + \frac{ (\rho_1^\alpha+\rho_0^\alpha) (\sigma_{\bj}^\alpha \tilde{\psi}_\bj^{\alpha,k}) }{2}  + \sum_{l=1}^{k-1} \frac{ (\rho_{l+1}^\alpha - \rho_{l-1}^\alpha) ( \sigma_{j_\alpha}^\alpha \tilde{\psi}_\bj^{\alpha,k-l} ) }{2}.
\end{align}
Then replacing \( \rho, u_{\b0,\bj},\bj \) by \( \tau, u_\bi,\bi \), respectively, and transforming all variables back to the time domain, we arrive at the PML equation~\eqref{eq_disPMLmaineq1}. The PML equation~\eqref{eq_disPMLmaineq2} can be derived by a similar procedure using the identity~\eqref{eq_tdw2}.

Finally, we derive the PML equations~\eqref{eq_disPMLmaineq3}--\eqref{eq_disPMLmaineq6} for auxiliary variables \( \tilde{\tilde{\psi}}, \bar{\tilde{\psi}}, \tilde{\bar{\psi}}, \bar{\bar{\psi}} \). By the identity~\eqref{eq_taukDw2}, for \( k_1,k_2\in\mathbb{N} \), it has 
\begin{align}
 	\ri \omega \tau_{k_1}^1 \tau_{k_2}^2 D_{1,2} u_{\b0,\bj} =&~  \tau_{k_2}^2 D_2 \left( \ri \omega \tau_{k_1}^1 D_1 u_{\b0,\bj} \right) \notag\\
 	=&~ \tau_{k_2}^2 D_2 \Bigg( \frac{ (\rho_{k_1+2}^1 - \rho_{k_1}^1) u_{\b0,\bj} }{2h} + \frac{ (\rho_1^1 + \rho_0^1) (\sigma_{\bj}^1 \tau_{k_1}^1 D_1 u_{\b0,\bj}) }{2}  \notag\\
 	& + \sum_{l=1}^{k_1} \frac{ ( \rho_{l+1}^1 - \rho_{l-1}^1 ) ( \sigma_{\bj}^1 \tau_{k_1-l}^1 D_1 u_{\b0,\bj} ) }{2}  \Bigg) \notag\\
 	=&~ \frac{ (\rho_{k_1+2}^1 - \rho_{k_1}^1) \tau_{k_2}^2 D_2 u_{\b0,\bj} }{2h} + \frac{ (\rho_1^1 + \rho_0^1) (\sigma_{\bj}^1 \tau_{k_1}^1 \tau_{k_2}^2 D_{1,2} u_{\b0,\bj}) }{2}  \notag\\
 	& + \sum_{l=1}^{k_1} \frac{ ( \rho_{l+1}^1 - \rho_{l-1}^1 ) ( \sigma_{\bj}^1 \tau_{k_1-l}^1 \tau_{k_2}^2 D_{1,2} u_{\b0,\bj} ) }{2},
\end{align}
which implies that for \( k_1,k_2\in\bbz^+ \),
\begin{align}
 	\ri \omega \tilde{\tilde{\psi}}_\bj^{k_1,k_2} =&~ \frac{ (\rho_{k_1+1}^1 - \rho_{k_1-1}^1)  \tilde{\psi}_\bj^{2,k_2}  }{2h} + \frac{ (\rho_1^1 + \rho_0^1) (\sigma_{\bj}^1 \tilde{\tilde{\psi}}_\bj^{k_1,k_2}) }{2}  + \sum_{l=1}^{k_1-1} \frac{ ( \rho_{l+1}^1 - \rho_{l-1}^1 ) ( \sigma_{\bj}^1 \tilde{\tilde{\psi}}_\bj^{k_1-l,k_2} ) }{2}. 
\end{align}
Replacing \( \rho^1 ,\bj \) by \( \tau^1,\bi \), respectively, and transforming all variables into the time domain, we derive the PML equation~\eqref{eq_disPMLmaineq3}.
As the proofs for PML equations \eqref{eq_disPMLmaineq4}, \eqref{eq_disPMLmaineq5}, and \eqref{eq_disPMLmaineq6} are similar, we omit the details for brevity.


The derivation implies that the following theorem holds.
\begin{theorem}	\label{th_main2}
	Let \( u(t,\bi) = e^{-\ri\omega t} \prod_{\alpha=1}^2 e^{\ri \kappa_\alpha^h i_\alpha h}  \) ( \( \kappa_\alpha^h\in\mathcal{K} \), \( \omega\neq0\in\mathbb{R} \) ) be a discrete traveling wave that satisfies the wave equation \eqref{eq_diswaveeq}. Then the extension of this wave in the PML domain, which is given by
	\begin{align*}
 	 	u_{\bi}(t) = e^{-\ri\omega t} \prod_{\alpha=1}^2 \eta(i_\alpha,\sigma^\alpha,\kappa_\alpha^h,\omega) e^{\ri \kappa_\alpha^h i_\alpha h}\ \mbox{with}\ \eta(i_\alpha,\sigma^\alpha,\kappa_\alpha^h,\omega)=  \prod_{l=0}^{i_\alpha-1} \frac{ 2+\ri\frac{\sigma^\alpha_l}{\omega}(1-e^{-\ri \kappa_\alpha^h h}) }{ 2+\ri\frac{\sigma^\alpha_l}{\omega}(1-e^{\ri \kappa_\alpha^h h}) },
	\end{align*}
	is the solution of the discrete PML equations~\eqref{eq_disPMLmaineq}--\eqref{eq_disPMLmaineq6}.
\end{theorem}

This means that the plane (or evanescent) waves of discrete wave equation~\eqref{eq_diswaveeq} do not reflect back to disrupt the solution in the physical domain as they travel into the PML layer. By \eqref{eq_estdampcoe} the wave amplitude decays exponentially along the positive \(  \mathrm{sgn}\omega \Re\kappa_\alpha^h \in (0, \pi/h) \) direction in the PML layer.

\section{Implement}  \label{sec_implement}
Practical implementations require a discretization scheme in time based on the sparsity of the system matrices and designing  to reduce numerical complexity. 

We truncate the PML into a finite thickness at \( i_\alpha=n_p \) and apply the homogeneous Dirichlet boundary on the linear system~\eqref{eq_disPMLmaineq}--\eqref{eq_disPMLmaineq6}, that is, \( u_\bi \) and auxiliary variables are set to be zero when \( i_\alpha> n_p \). This may give rise to a residual wave with wave number \( -\bar{\kappa}^h_\alpha=-\Re(\kappa^h_\alpha)+\ri \Im(\kappa_\alpha^h) \) which re-enters the layer along the negative \(  -\mathrm{sgn}\omega \Re\kappa_\alpha^h \) direction. By the time the reflected wave propagates through the wave layer and reaches the physical domain, its amplitude has been suppressed by the cumulative amount of \( |\eta(j_\alpha,\sigma^\alpha,\kappa_\alpha^h,\omega)||\eta(j_\alpha,\sigma^\alpha,-\bar{\kappa}_\alpha^h,\omega)| \) which is damped exponentially when \( \omega\in\mathbb{R}, \sigma^\alpha>0, \kappa_\alpha^h\in\mathcal{K} \) (cf. \cite[Theorem 3.2]{chern2019reflectionless}).

For the sake of conciseness, we assume that the coefficents \( a_\bk \) may not be zero only if \( |k_1|\leq p \) and \( |k_2|\leq p \). The linear system~\eqref{eq_disPMLmaineq}--\eqref{eq_disPMLmaineq6} can be  be written as
\begin{align}
 	\partial_t^2 U =&~ \mathcal{M}_0 U + \sum_{\alpha=1}^2 \sum_{k=1}^p \left( \mathcal{A}_0^{\alpha,k} \tilde{\Psi}^{\alpha,k} + \mathcal{B}_0^{\alpha,k} \bar{\Psi}^{\alpha,k} \right) \notag\\
 	& + \sum_{k_1=1}^p \sum_{k_2=1}^p \left( \mathcal{C}_0^{k_1,k_2} \tilde{\tilde{\Psi}}^{k_1,k_2} + \mathcal{D}_0^{k_1,k_2} \tilde{\bar{\Psi}}^{k_1,k_2} + \mathcal{E}_0^{k_1,k_2} \bar{\tilde{\Psi}}^{k_1,k_2} + \mathcal{F}_0^{k_1,k_2} \bar{\bar{\Psi}}^{k_1,k_2} \right), \label{eq_matrixeq0}\\
 	\partial_t \tilde{\Psi}^{\alpha,k} =&~ \mathcal{M}_1^{\alpha,k} U + \mathcal{A}_1^{\alpha,k} \tilde{\Psi}^{\alpha,k} + \sum_{l=1}^{k-1} A_1^{\alpha,l} \tilde{\Psi}^{\alpha,l},\quad \alpha=1,2,k=1,2,\cdots,p, \label{eq_matrixeq1}\\
 	\partial_t \bar{\Psi}^{\alpha,k} =&~ \mathcal{M}_2^{\alpha,k} U + \mathcal{B}_2^{\alpha,k} \bar{\Psi}^{\alpha,k} + \sum_{l=1}^{k-1} B_2^{\alpha,l} \tilde{\Psi}^{\alpha,l}, \quad \alpha=1,2,k=1,2,\cdots,p,\label{eq_matrixeq2} \\
 	\partial_t \tilde{\tilde{\Psi}}^\bk =&~ \mathcal{A}_3^{k_1} \tilde{\Psi}^{2,k_2} + \mathcal{C}_3^0 \tilde{\tilde{\Psi}}^\bk + \sum_{l=1}^{k_1-1} \mathcal{C}_3^l \tilde{\tilde{\Psi}}^{k_1-l,k_2},\quad 1\leq k_1,k_2 \leq p, \label{eq_matrixeq3}\\
 	\partial_t \tilde{\bar{\Psi}}^\bk =&~ \mathcal{A}_4^{k_1} \tilde{\Psi}^{2,k_2} + \mathcal{D}_4^0 \tilde{\bar{\Psi}}^\bk + \sum_{l=1}^{k_1-1} \mathcal{D}_4^l \tilde{\bar{\Psi}}^{k_1-l,k_2}, \quad 1\leq k_1,k_2 \leq p, \label{eq_matrixeq4}\\
 	\partial_t \bar{\tilde{\Psi}}^\bk =&~ \mathcal{B}_5^{k_1} \bar{\Psi}^{2,k_2} + \mathcal{E}_5^0 \bar{\tilde{\Psi}}^\bk + \sum_{l=1}^{k_1-1} \mathcal{E}_5^l \bar{\tilde{\Psi}}^{k_1-l,k_2}, \quad 1\leq k_1,k_2 \leq p, \label{eq_matrixeq5}\\
 	\partial_t \bar{\bar{\Psi}}^\bk =&~ \mathcal{B}_6^{k_1} \bar{\Psi}^{2,k_2} + \mathcal{F}_6^0 \bar{\bar{\Psi}}^\bk + \sum_{l=1}^{k_1-1} \mathcal{F}_6^l \bar{\bar{\Psi}}^{k_1-l,k_2},\quad 1\leq k_1,k_2 \leq p, \label{eq_matrixeq6}
\end{align}
where the column vectors \( U,\tilde{\Psi}^{\alpha,k},\bar{\Psi}^{\alpha,k},\tilde{\tilde{\Psi}}^\bk,\tilde{\bar{\Psi}}^\bk,\bar{\tilde{\Psi}}^\bk,\bar{\bar{\Psi}}^\bk \) are generated by the degrees of freedom of elements \( u_\bi, \tilde{\psi}^{\alpha,k}_\bi,\bar{\psi}^{\alpha,k}_\bi,\tilde{\tilde{\psi}}^\bk_\bi,\tilde{\bar{\psi}}^\bk_\bi,\bar{\tilde{\psi}}^\bk_\bi,\bar{\bar{\psi}}^\bk_\bi \), respectively. The order of elements allocated into the column vector follows the order of \( i_2 \) column, and then moving the next \( i_2+1 \) column. Note that the particular ordering of the nodes affects the matrix structure corresponding to the discrete system but does not affect the numerical solution.

We remark again that the auxiliary variables \( \tilde{\Psi}^{\alpha,k} \) (resp. \( \bar{\Psi}^{\alpha,k} \)) and their evolution equations~\eqref{eq_matrixeq1} (resp. \eqref{eq_matrixeq2}) only need to be supported on the PML layer \( \{ \ 0\leq i_\alpha\leq n_p \} \). The auxiliary variables \( \tilde{\tilde{\Psi}}^\bk \) (resp. \( \tilde{\bar{\Psi}}^\bk,\bar{\tilde{\Psi}}^\bk,\bar{\bar{\Psi}}^\bk \)) and their evolution equations~\eqref{eq_matrixeq3} (resp. \eqref{eq_matrixeq4}--\eqref{eq_matrixeq6}) are only on the PML corner \( \{ \ 0\leq i_1\leq n_p \} \cap \{\ 0\leq i_2\leq n_p \} \)

The coefficient matrices \( \mathcal{M},\mathcal{A},\cdots,\mathcal{F} \) can be inferred from \eqref{eq_disPMLmaineq}--\eqref{eq_disPMLmaineq6}. It can be observed that
\begin{itemize}
	\item \( \mathcal{A}_3^{k_1} = \mathcal{B}_5^{k_1} \), \( \mathcal{A}_4^{k_1} = \mathcal{B}_6^{k_1} \), \( \mathcal{C}_3^l=\mathcal{E}_5^l \) and \( \mathcal{D}_4^l=\mathcal{F}_6^l \);
	\item  \( \mathcal{A}_1^{\alpha,k} \) and \( \mathcal{C}_3^0 = \mathcal{E}_5^0 \) are upper triangular matrices;
	\item \( \mathcal{B}_2^{\alpha,k} \) and \( \mathcal{D}_4^0 = \mathcal{F}_6^0  \) are lower triangular matrices.
\end{itemize}
If we introduce vectors
\begin{align*}
 	\tilde{\Psi} =&~ [ \tilde{\Psi}^{1,p}; \tilde{\Psi}^{1,p-1}; \cdots ; \tilde{\Psi}^{1,1}; \cdots; \tilde{\Psi}^{2,p}; \tilde{\Psi}^{2,p-1}; \cdots ; \tilde{\Psi}^{2,1} ], \\
 	\bar{\Psi} =&~ [ \bar{\Psi}^{1,1}; \bar{\Psi}^{1,2}; \cdots ; \bar{\Psi}^{1,p}; \cdots; \bar{\Psi}^{2,1}; \bar{\Psi}^{2,2}; \cdots ; \bar{\Psi}^{2,p} ],\\
 	\tilde{\tilde{\Psi}} =&~ [ \tilde{\tilde{\Psi}}^{p,1}; \tilde{\tilde{\Psi}}^{p-1,1}; \cdots; \tilde{\tilde{\Psi}}^{1,1};\cdots;\tilde{\tilde{\Psi}}^{p,2}; \tilde{\tilde{\Psi}}^{p-1,2}; \cdots; \tilde{\tilde{\Psi}}^{1,2};\cdots;\tilde{\tilde{\Psi}}^{p,p}; \tilde{\tilde{\Psi}}^{p-1,p}; \cdots; \tilde{\tilde{\Psi}}^{1,p}], \\
 	\tilde{\bar{\Psi}} =&~ [ \tilde{\bar{\Psi}}^{1,1}; \tilde{\bar{\Psi}}^{1,2}; \cdots; \tilde{\bar{\Psi}}^{1,p};\cdots;\tilde{\bar{\Psi}}^{2,1}; \tilde{\bar{\Psi}}^{2,2}; \cdots; \tilde{\bar{\Psi}}^{2,p};\cdots;\tilde{\bar{\Psi}}^{p,1}; \tilde{\bar{\Psi}}^{p,2}; \cdots; \tilde{\bar{\Psi}}^{p,p}], \\
 	\bar{\tilde{\Psi}} =&~ [ \bar{\tilde{\Psi}}^{p,1}; \bar{\tilde{\Psi}}^{p-1,1}; \cdots; \bar{\tilde{\Psi}}^{1,1};\cdots;\bar{\tilde{\Psi}}^{p,2}; \bar{\tilde{\Psi}}^{p-1,2}; \cdots; \bar{\tilde{\Psi}}^{1,2};\cdots;\bar{\tilde{\Psi}}^{p,p}; \bar{\tilde{\Psi}}^{p-1,p}; \cdots; \bar{\tilde{\Psi}}^{1,p}], \\
 	\bar{\bar{\Psi}} =&~ [ \bar{\bar{\Psi}}^{1,1}; \bar{\bar{\Psi}}^{1,2}; \cdots; \bar{\bar{\Psi}}^{1,p};\cdots;\bar{\bar{\Psi}}^{2,1}; \bar{\bar{\Psi}}^{2,2}; \cdots; \bar{\bar{\Psi}}^{2,p};\cdots;\bar{\bar{\Psi}}^{p,1}; \bar{\bar{\Psi}}^{p,2}; \cdots; \bar{\bar{\Psi}}^{p,p}].
\end{align*}
\eqref{eq_matrixeq0}--\eqref{eq_matrixeq6} can be expressed in a further compact form
\begin{align*}
 	\partial_t^2 U =&~ \mathcal{M}_0 U +  \mathcal{A}_0 \tilde{\Psi} + \mathcal{B}_0 \bar{\Psi} + \mathcal{C}_0 \tilde{\tilde{\Psi}} + \mathcal{D}_0 \tilde{\bar{\Psi}} + \mathcal{E}_0 \bar{\tilde{\Psi}} + \mathcal{F}_0 \bar{\bar{\Psi}},\\
 	\partial_t \tilde{\Psi} =&~ \mathcal{M}_1 U + \mathcal{A}_1 \tilde{\Psi},\quad \partial_t \bar{\Psi} = \mathcal{M}_2 U + \mathcal{B}_2 \bar{\Psi},\\
 	\partial_t \tilde{\tilde{\Psi}} =&~ \mathcal{A}_3 \tilde{\Psi}^{2} + \mathcal{C}_3 \tilde{\tilde{\Psi}},\quad \partial_t \tilde{\bar{\Psi}} = \mathcal{A}_4 \tilde{\Psi}^{2} + \mathcal{D}_4 \tilde{\bar{\Psi}}, \\
 	\partial_t \bar{\tilde{\Psi}} =&~ \mathcal{B}_5 \bar{\Psi}^{2} + \mathcal{E}_5 \bar{\tilde{\Psi}},\quad 	\partial_t \bar{\bar{\Psi}} = \mathcal{B}_6 \bar{\Psi}^{2} + \mathcal{F}_6 \bar{\bar{\Psi}}.
\end{align*}
Because \( \mathcal{A}_1 \) and \( \mathcal{C}_3=\mathcal{E}_5 \) are upper triangular matrices, and \( \mathcal{B}_2 \) and \( \mathcal{D}_4=\mathcal{F}_6 \) are upper triangular matrices, we use the explicit Verlet-type algorithm to solve the semi-discrete linear system efficiently. We omit the details to save space.

\section{Numerial tests} \label{sec_num}
In this section, we demonstrate the effectiveness and efficiency of the discrete PML for the PD scalar wave-type equation. In the following examples, the physical domain is chosen as \( (-1,1)\times(-1,1) \). Denote by $h$ the mesh size and by \( \Omega=\{\bi\in\bbz^2|\ |i_1|\leq n,\ |i_2|\leq n\} \)  with $n=1/h$. We simply set the layer thickness as \( 4h \) (i.e. \( n_p=4 \)) to minimize the number of auxiliary functions and reduce extra computational cost. Initial conditions are given as a pulse, i.e., \( u(\bx,0) = e^{-40|\bx|^2}, \partial_t u(\bx,0) = 0 \).

\textbf{Example 1.}
We first consider the Gaussian kernel
\begin{align}
 	\gamma(\bx-\bx') = \frac{4}{\pi\epsilon^4} e^{-\frac{|\bx-\bx'|^2}{\epsilon^2}}.
\end{align}
This kernel was simulated in \cite{DuHanZhangZheng,wildman2012a,du2018nonlocal} to show the eﬃciency of nonlocal ABCs. The Gaussian kernel does not have a strict bond \( \delta \), but exponentially decays to zero. We set the cutoff as \( 10^{-7} \) such that \( \delta=\frac14 \). In this situation, the kernel is smooth and does not result in a large number of bonds. 

At first, it is valuable to investigate how to choose an economic layer thickness and suitable PML coefficients \( \sigma_l^\alpha \) for the computational decision, and then present a practical advice on the choice of \( \sigma_l^\alpha \). To the end, it is straightforward to verify that in the asymptotical assumption \( |\kappa_\alpha^h h|<1 \), it holds
\begin{align}
	\left| \frac{ 2+\ri\frac{\sigma^\alpha_l}{\omega}(1-e^{-\ri \kappa_\alpha^h h}) }{ 2+\ri\frac{\sigma^\alpha_l}{\omega}(1-e^{\ri \kappa_\alpha^h h}) } \right| = \left| \frac{ 2 - \frac{\kappa_\alpha^h}{\omega} \sigma_l^\alpha h ( 1 - \frac{\ri \kappa_\alpha^h h}{2!} + \frac{(\ri \kappa_\alpha^h h)^2}{3!} + \cdots) }{ 2 + \frac{\kappa_\alpha^h}{\omega} \sigma_l^\alpha h ( 1 + \frac{\ri \kappa_\alpha^h h}{2!} + \frac{(\ri \kappa_\alpha^h h)^2}{3!} + \cdots)  } \right| \leq \left| \frac{ 2 - c_1|\frac{\kappa_\alpha^h}{\omega}| \sigma_l^\alpha h  }{ 2 + c_2|\frac{\kappa_\alpha^h}{\omega}| \sigma_l^\alpha h } \right|,
\end{align}
where \( c_1 \) and \( c_2 \) are positive constants depending on \( \kappa_\alpha^h h \). Therefore,  the asymptotical decay rate only depends  on \( \sigma_l^\alpha h \) for a fixed incident angle \( \cos^{-1} (\Re\kappa_\alpha^h/\omega) \). Consequently, we  can choose the PML coefficients as \( \sigma_l^\alpha = O(1/h) \) when \( |l|\geq n \). We now investigate how to choose the value of \( \sigma_l^\alpha h \) for a good performance by setting \( \sigma^\alpha_l=\sigma_0 \) for all \( \alpha=1,2 \), and denote the decay rate of PML  by
\[
	\mu(\bka^h,\sigma_0)= \frac{2+\ri\frac{\sigma_0}{\omega}(1-e^{-\ri \kappa_\alpha^h h})}{2+\ri\frac{\sigma_0}{\omega}(1-e^{\ri \kappa_\alpha^h h})}.
\]
Fig.~\ref{fig_deamprate} illustrates the decay rate \( \mu(\bka^h,\sigma_0) \) with respect to \( \kappa_1^h \) and \( \sigma_0 \) for fixed \( \kappa_2^h=10,20,40 \). The left graph of Fig.~\ref{fig_damprate2} shows the decay rate \( \mu(\bka^h,\sigma_0) \) with respect to \( \sigma_0 \) for different wave number \( \bka^h \). These results show that \( \sigma_0 \) can be chosen between \( (1/h,5/h) \) in order to minimize \( |\mu(\bka^h,\sigma_0)| \). Furthermore, we consider the numerical reflection measured directly by the re-entering residual wave of max norm $|U-U_{\mathrm{ref}}|_\infty$ in the physical domain against a reference solution $U_{\mathrm{ref}}$ solving discrete wave equation~\eqref{eq_diswaveeq} over a larger domain \( (-8,8)\times(-8,8) \). The right graph of Fig.~\ref{fig_damprate2} plots the max amplitude of the re-entering residual wave, showing that the setup with \( \sigma_0=2/h \) gives a better performance for suppressing the residual wave.

\begin{figure}
	\centering
 	\includegraphics[width=0.3\textwidth]{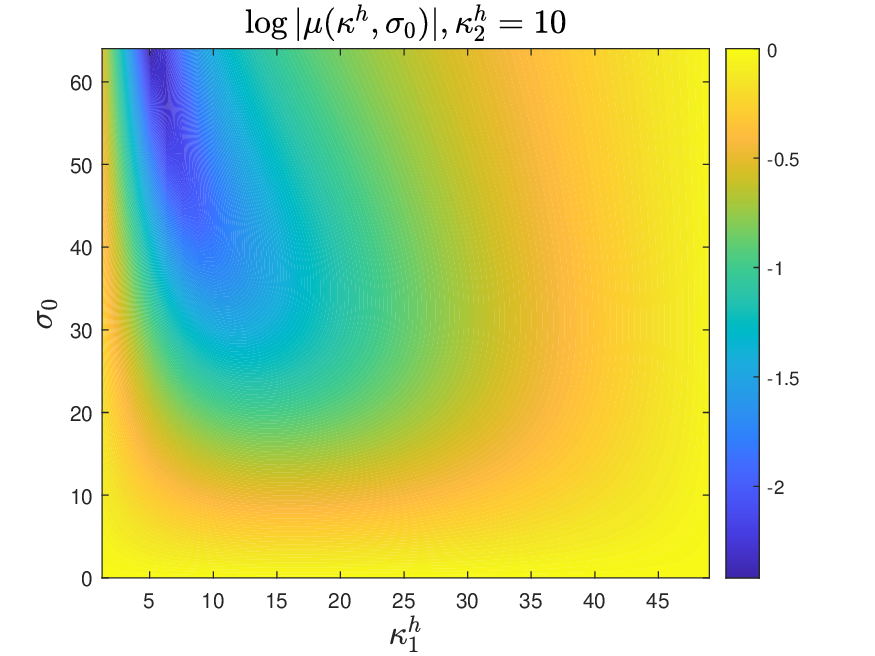}
 	\includegraphics[width=0.3\textwidth]{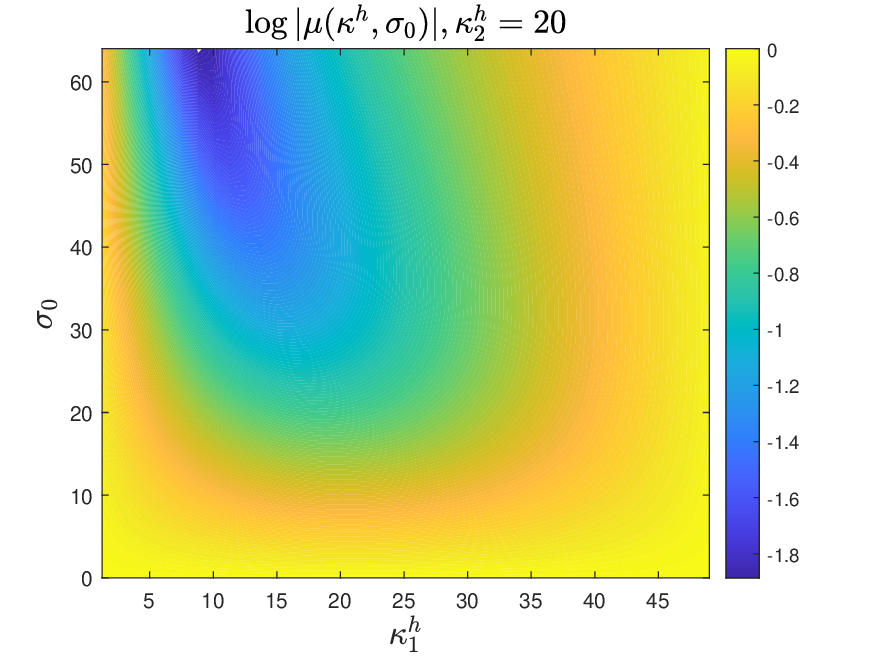}
 	\includegraphics[width=0.3\textwidth]{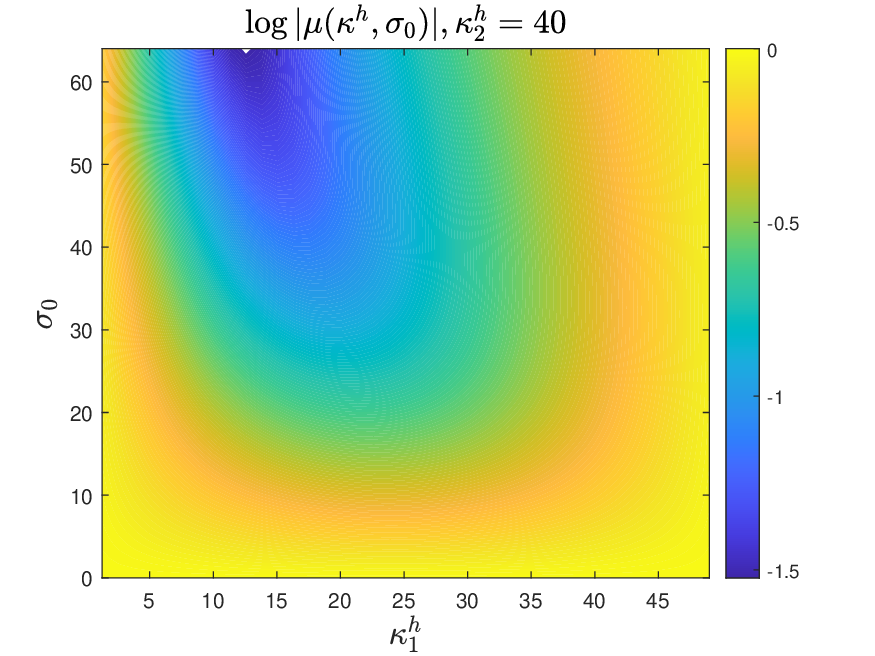}
 	\caption{(Example 1.) The decay rate \( \mu(\bka^h,\sigma_0) \) as a function of \( \sigma_0 \) and \( \kappa_1^h \) for fixed \( \kappa_2^h\). The mesh size is \( h=1/16 \) and the horizon of the kernel is \(\delta=1/4 \). } \label{fig_deamprate}
\end{figure}

\begin{figure}
	\centering
 	\includegraphics[width=0.4\textwidth]{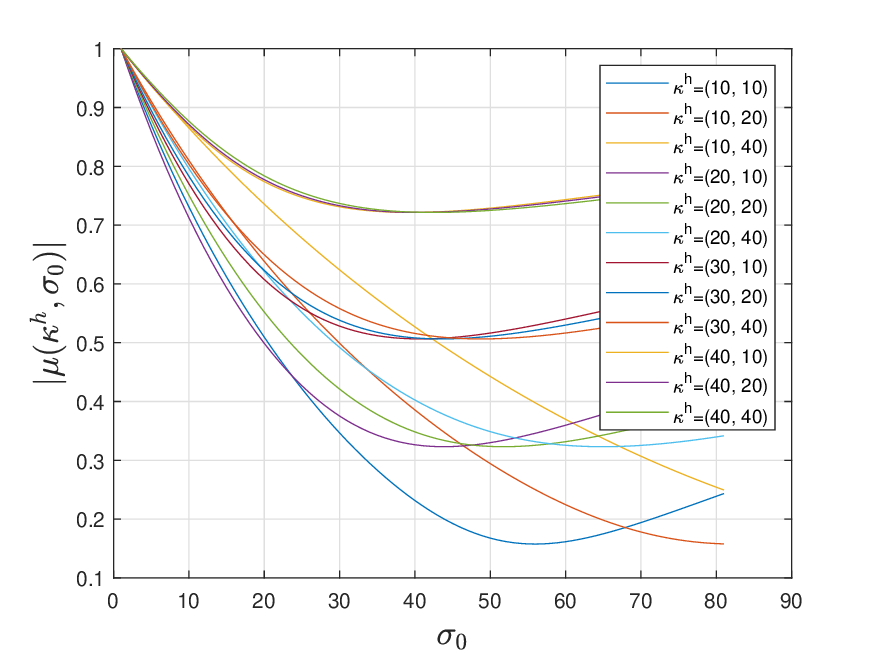}
 	\includegraphics[width=0.4\textwidth]{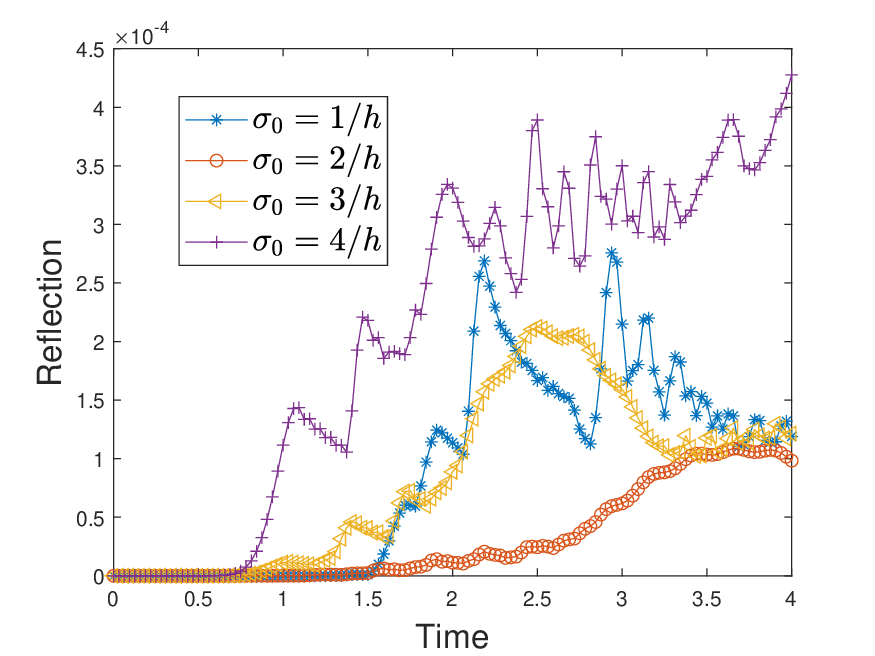}
 	\caption{(Example 1.) Left: The decay rate \( \mu(\bka^h,\sigma_0) \) as a function of \( \sigma_0 \) for different \( \bka^h \). Right: The reflection error \( \max|U-U_{\mathrm{ref}}| \) over \( \Omega \) with mesh size 	\( h=1/16 \). \( U \) and \( U_{\mathrm{ref}} \) are solutions to the problem with horizon \( \delta=1/4 \).} \label{fig_damprate2}
\end{figure}

Fig.~\ref{fig_ex1solutions} presents simulation results at different time instances. The whole computational domain is discretized with mesh size \( h=1/16 \) and time step \( \Delta t = h/32 \). The reference solution is computed over a larger spatial
domain \( (-16,16)\times(-16,16) \) to avoid any boundary effect with a refined mesh, say \( h=1/64 \) and \( \Delta t=h^2 \). It is observed that the PML solution is in good agreement with the reference solution in the physical domain. Waves traveling into the PML are allowed to be transmitted out without any significant reflections back into the computational domain, and are damped exponentially.

To show the accuracy of the quadrature-based finite difference scheme, Fig.~\ref{fig_error} presents the errors in both $l^2$-norm and max-norm between the PML solution and the reference solution in the physical domain at $t=1.5$. The reference solutions are obtained by mesh size \( h=1/64 \). This suggests a convergence rate of at least $2$.

\begin{figure}
	\centering
 	\includegraphics[width=0.3\textwidth]{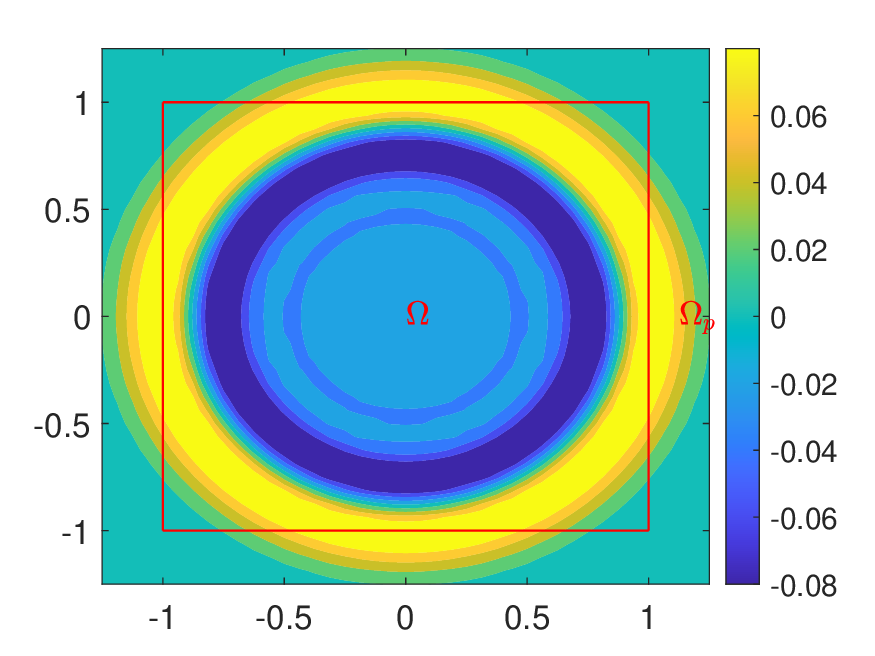}
 	\includegraphics[width=0.3\textwidth]{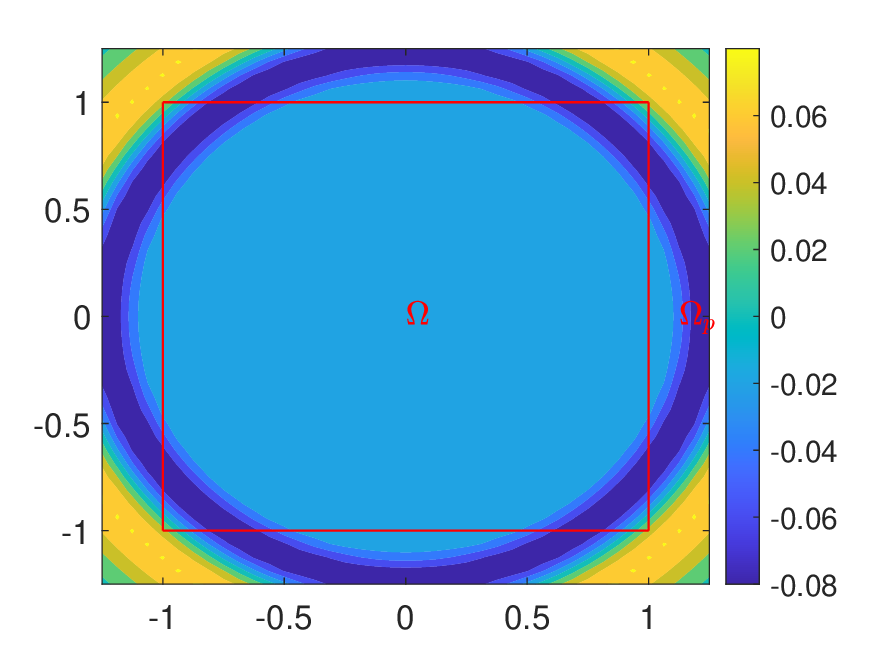}
 	\includegraphics[width=0.3\textwidth]{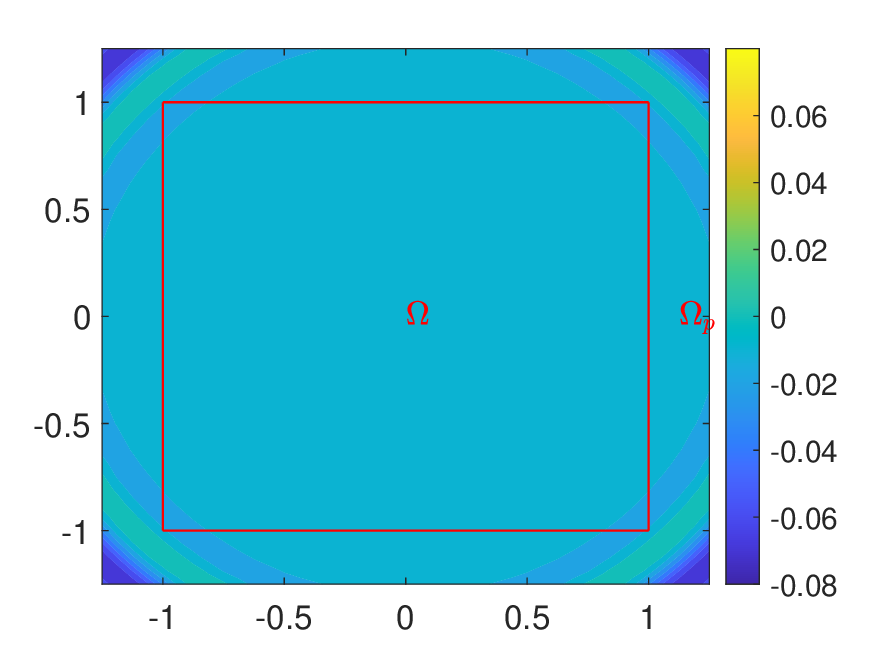}
 	\includegraphics[width=0.3\textwidth]{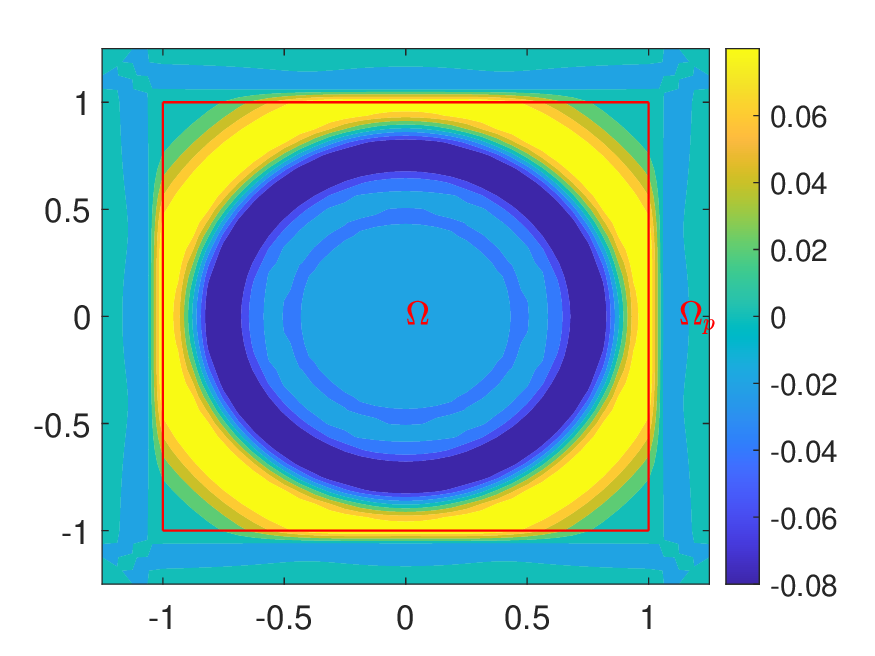}
 	\includegraphics[width=0.3\textwidth]{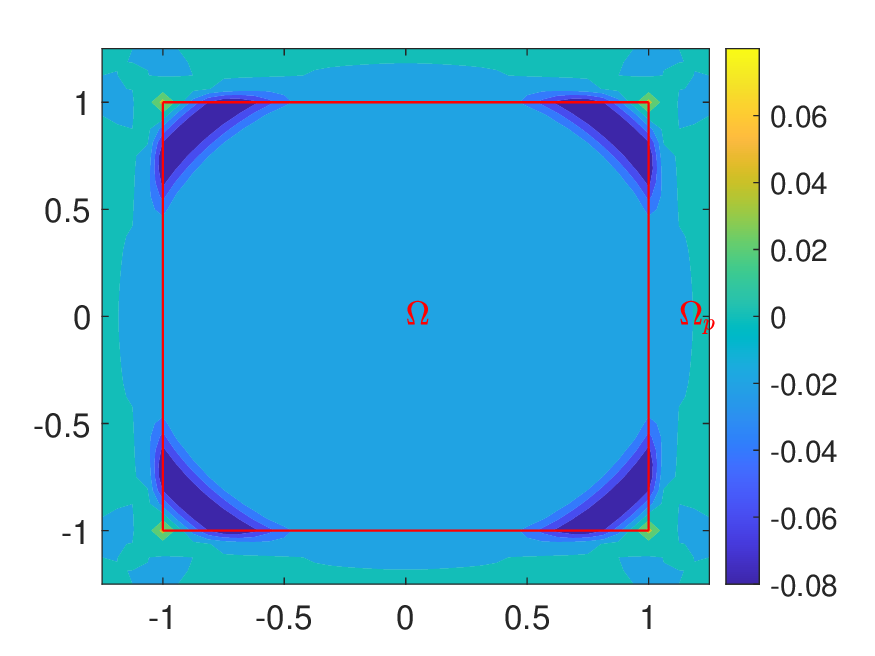}
 	\includegraphics[width=0.3\textwidth]{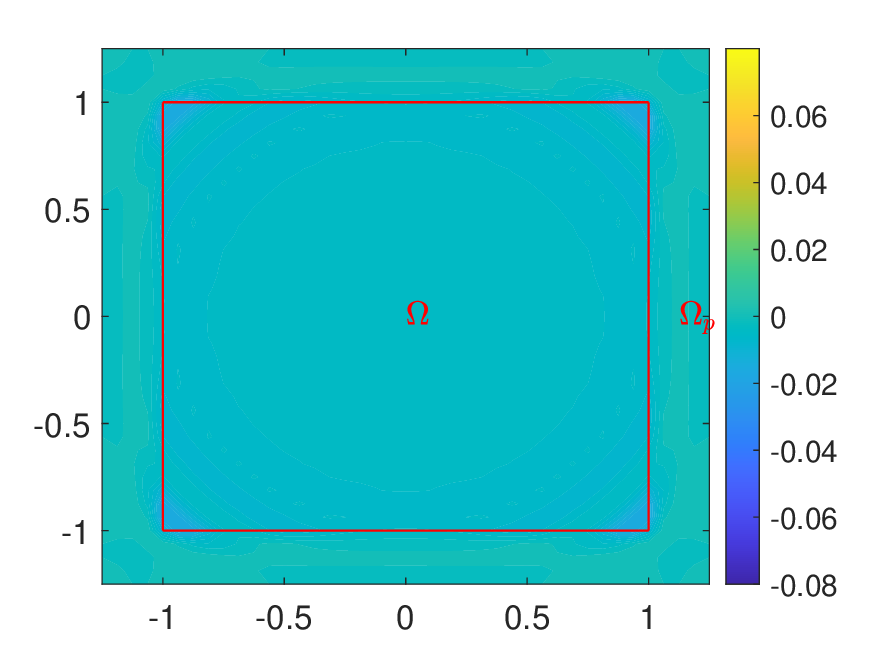}
 	\caption{(Example 1.) Waves in the PML (shown as grids) are damped exponentially. The solution in the near field at different time instants. Top, from left to right: reference solutions at \( t=1,1.5,2 \). Bottom, from left to right: PML solutions at \( t=1,1.5,2 \).} \label{fig_ex1solutions}
\end{figure}

\begin{figure}
	\centering
	\includegraphics[width=0.4\textwidth]{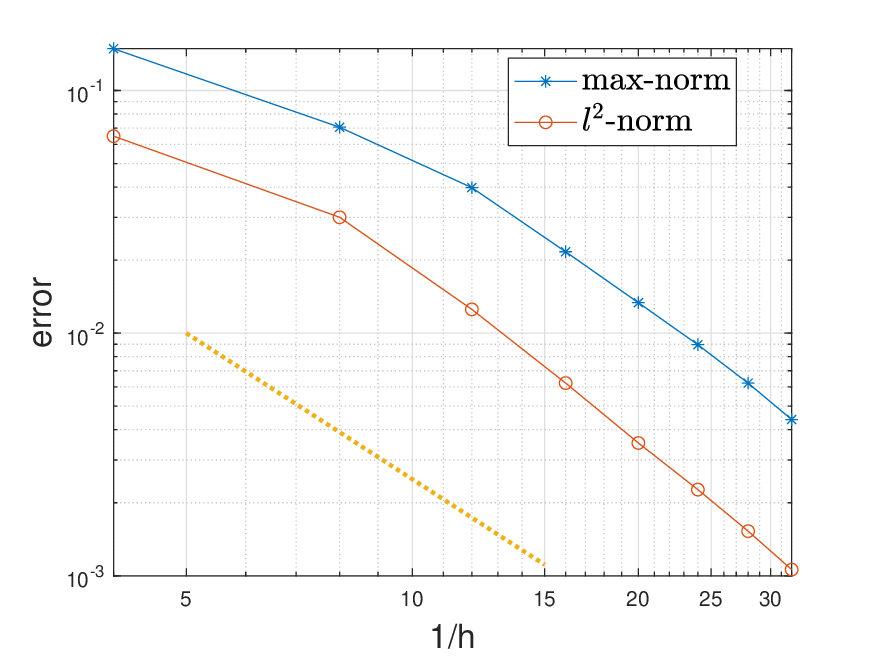}
 	\caption{(Example 1.)  The errors in \( l^2 \)-norm and max-norm between the PML solution and the reference solution at \( t =1.5 \). The dash line is a reference line with slope $-2$. } \label{fig_error}
\end{figure}

Next, we check the numerical stability of Eqs.~\eqref{eq_disPMLmaineq}--\eqref{eq_disPMLmaineq6} truncated with the homogeneous Dirichlet boundary. The stability is numerically confirmed in Fig.~\ref{fig_stab} by measuring the long time behavior of the energy norm of solution. The energy of numerical solution is defined by
\begin{align}
 	E(t_n) = \frac{1}{2|\Omega|} \sum_{\bi\in\Omega} \left| \frac{u_{\bi}(t_{n+1})-u_\bi(t_{n-1})}{2\Delta t} \right|^2 + \frac{1}{2|\Omega_\delta|} \sum_{\bi\in\Omega_\delta} \left| \sum_\bk a_\bk u_{\bi+\bk}(t_n) \right|^2,
\end{align}
where \( \Omega_\delta:=\{\bi\in\Omega|\ \bi+\bk\in\Omega\ \mbox{for all}\ a_\bk\neq0\} \). The PML solution and reference solution are solved by mesh  size \( h=1/16 \).


\begin{figure}
	\centering
	\includegraphics[width=0.4\textwidth]{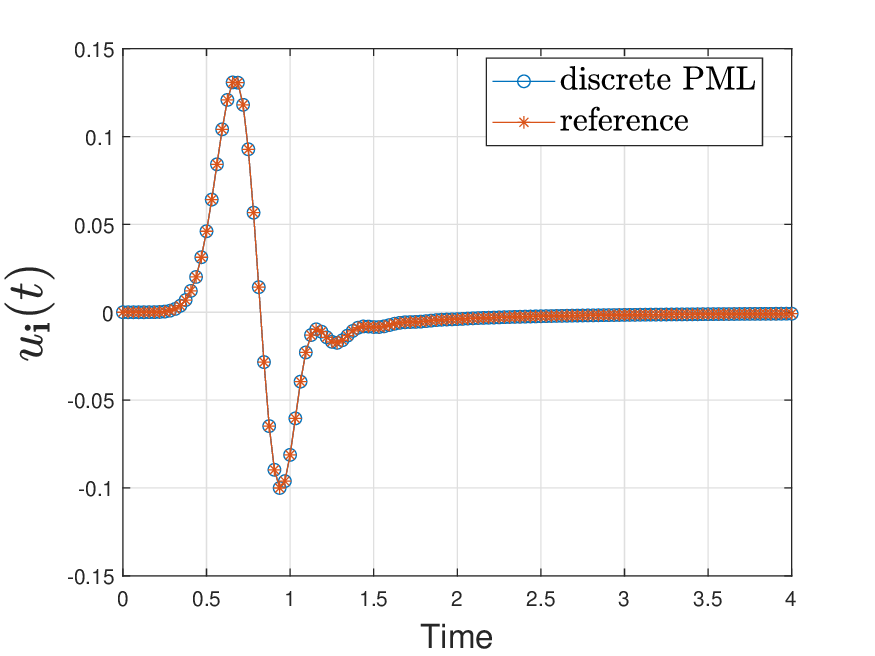}
	\includegraphics[width=0.4\textwidth]{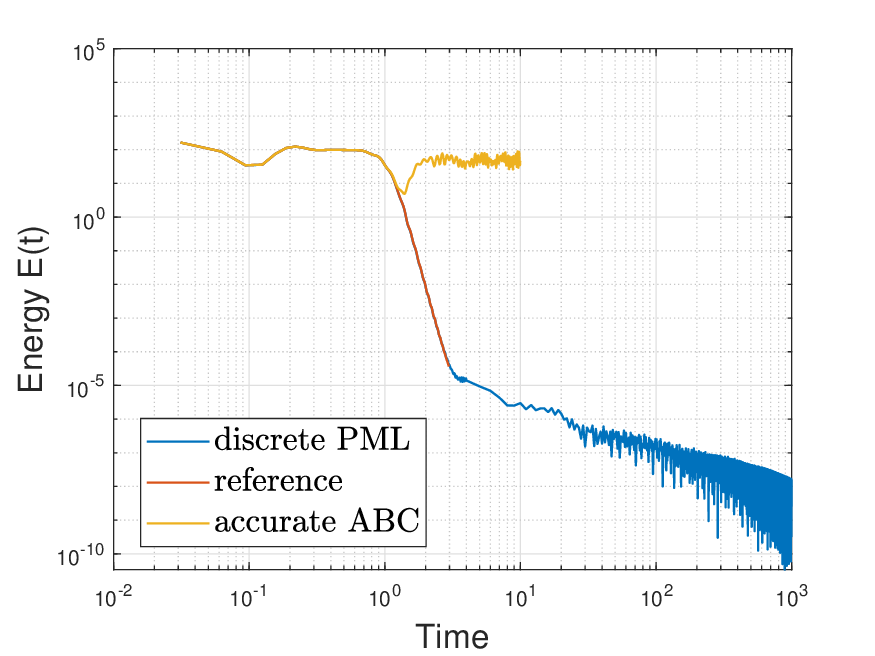}
	\caption{(Example 1.) Left: Time evolution of the solutions at \( \bx=(0.5,0.5) \). Right: The energy evolution of the discrete PML and reference solutions. Numerical solutions are solved under the mesh size \( h=1/16 \). } \label{fig_stab}
\end{figure}

Finally, we give a CPU time comparison between the discrete PML and the accurate ABCs (cf. \cite{pang2022accurate,pang2023accurate,ji2021accurate}) in Table~\ref{tab_ex1_1}. The computations were done on Intel 8358 Processor. The results clearly show that the computational efficiency of the discrete PML method is significantly improved comparing with that of accurate ABC.


%
\begin{table}
\centering
\begin{tabular}{l c c c c c c}
\hline
Method/($h$,$t$) 	& (1/8,4)	&  (1/8,8) 		&  (1/8,16) 	&  (1/16,4) 	&  (1/16,8) 	&  (1/16,16)		\\ \hline
discrete PML		& 0.86	& 	1.76		&	3.15		&	4.84		&	9.38		&	19.16		\\
accurate ABC  			&547.40	&  	1062.04	&	1937.77	&	2891.31 	& 6266.84		&	12772.83		\\ \hline 
\end{tabular}
\caption{(Example 1.) Comparison of CPU time (in seconds) between the discrete PML and the accurate absorbing boundary condition for different mesh size \( h \) and final time \( t \). }\label{tab_ex1_1}
\end{table}

\textbf{Example 2.} We consider the kernel in the form of 
\begin{align}
 	\gamma(\bx-\bx') = \frac{4}{\pi\delta^2 |\bx-\bx'|^2} 1_{\mathcal{H}_{\bx-\bx'}},
\end{align}
with the horizon parameter \( \delta=1/4 \), where \( 1_{\mathcal{H}_{\bx-\bx'}} \) is a Heaviside function supported in \( \mathcal{H}_{\bx-\bx'} \). In the simulations,  we simply choose the PML coefficients \( \sigma_l^\alpha = 2/h \).

Fig.~\ref{fig_ex3solutions} presents the simulation results at different time instances. The whole computational domain is discretized by \( h=1/16 \) and \( \Delta t=h/32 \). One can also observe that the PML solution is almost identical to the reference solution in the physical domain, and the wave enters the PML without any significant reflection and decays exponentially.

\begin{figure}[t]
	\centering
 	\includegraphics[width=0.3\textwidth]{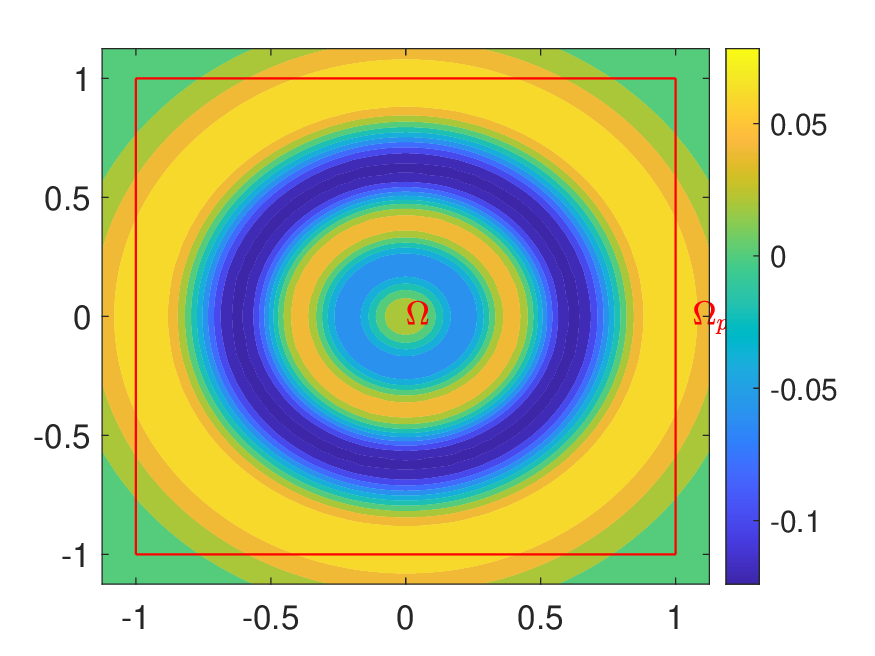}
 	\includegraphics[width=0.3\textwidth]{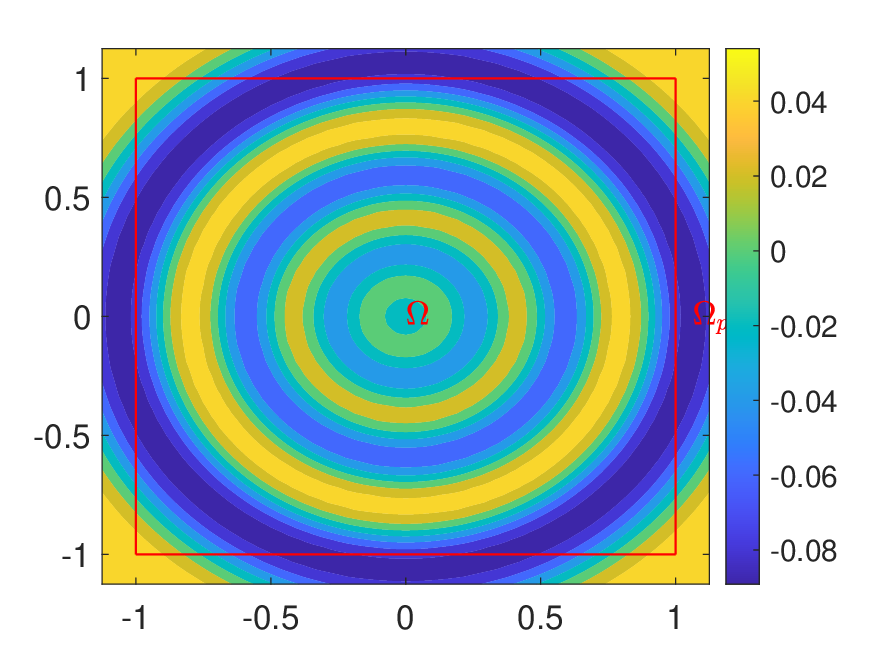}
 	\includegraphics[width=0.3\textwidth]{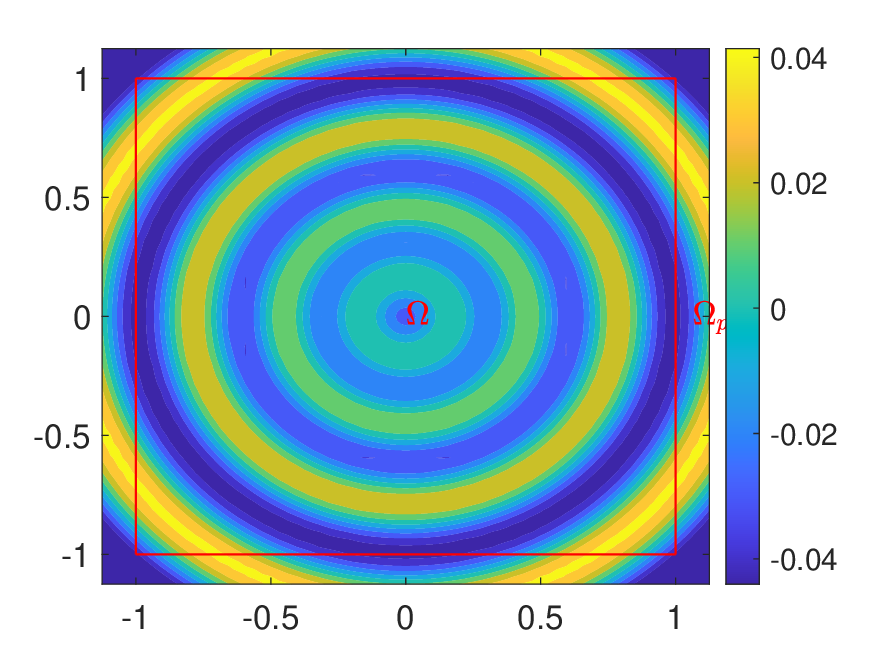}
 	\includegraphics[width=0.3\textwidth]{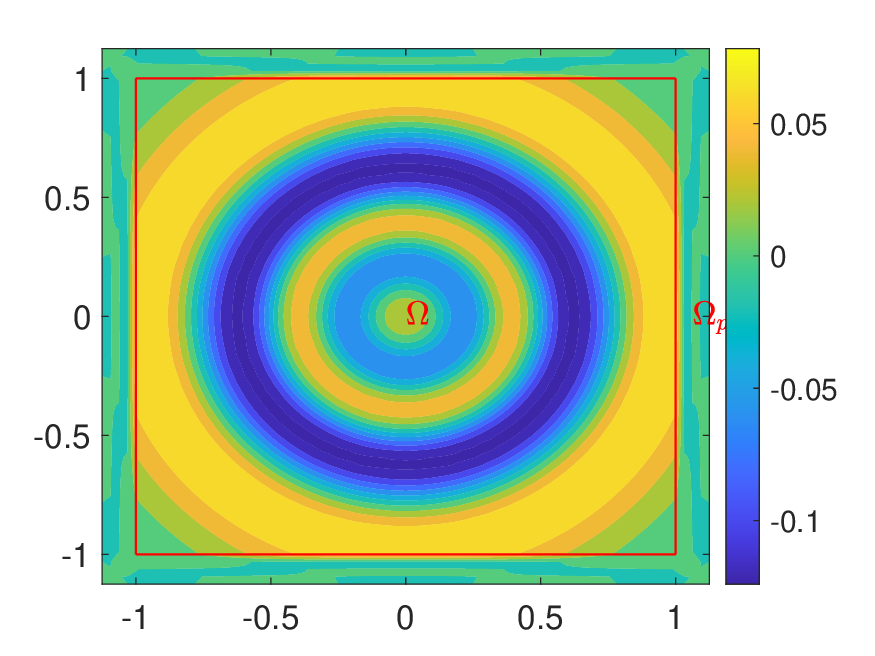}
 	\includegraphics[width=0.3\textwidth]{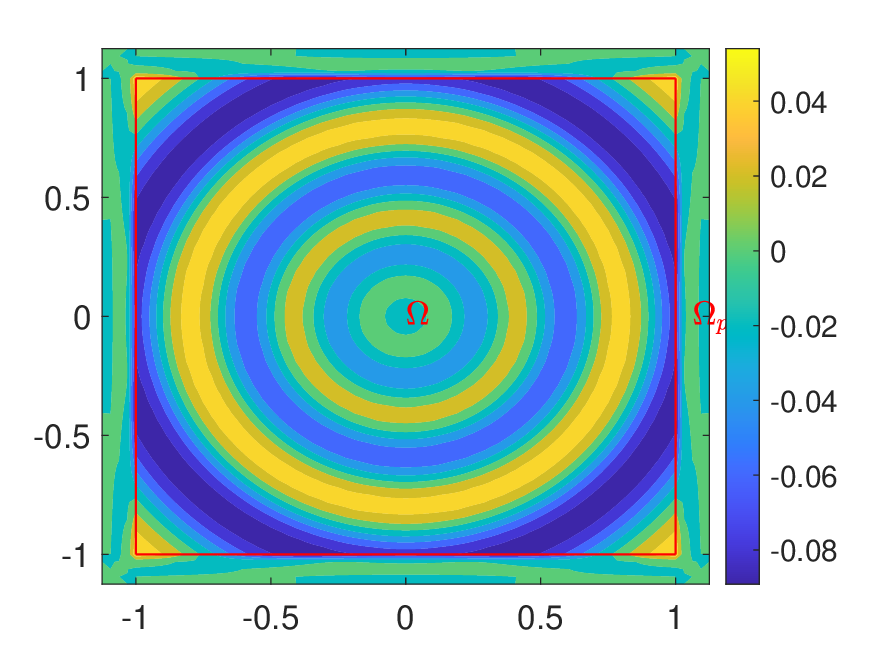}
 	\includegraphics[width=0.3\textwidth]{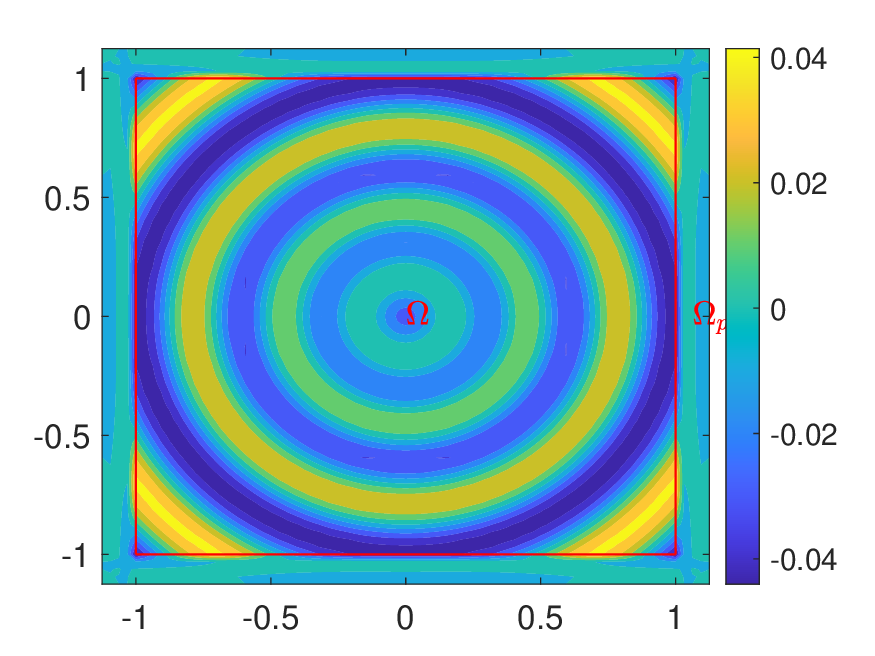}
 	\caption{(Example 2.) Waves in the PML (shown as grids) are damped exponentially. The solution in the near field at different time instants. Top, from left to right: reference solutions at \( t=1,1.5,2 \). Bottom, from left to right: PML solutions at \( t=1,1.5,2 \).} \label{fig_ex3solutions}
\end{figure}

Finally, we  check the numerical stability by measuring the long time behavior of the energy norm of solutions  in Fig.~\ref{fig_ex3stab}. In additoin, the results on the accuracy of the quadrature-based finite difference scheme and the comparison of computation performance are almost the same as those in Example 1, we omit them here. 


\begin{figure}[t]
	\centering
	\includegraphics[width=0.4\textwidth]{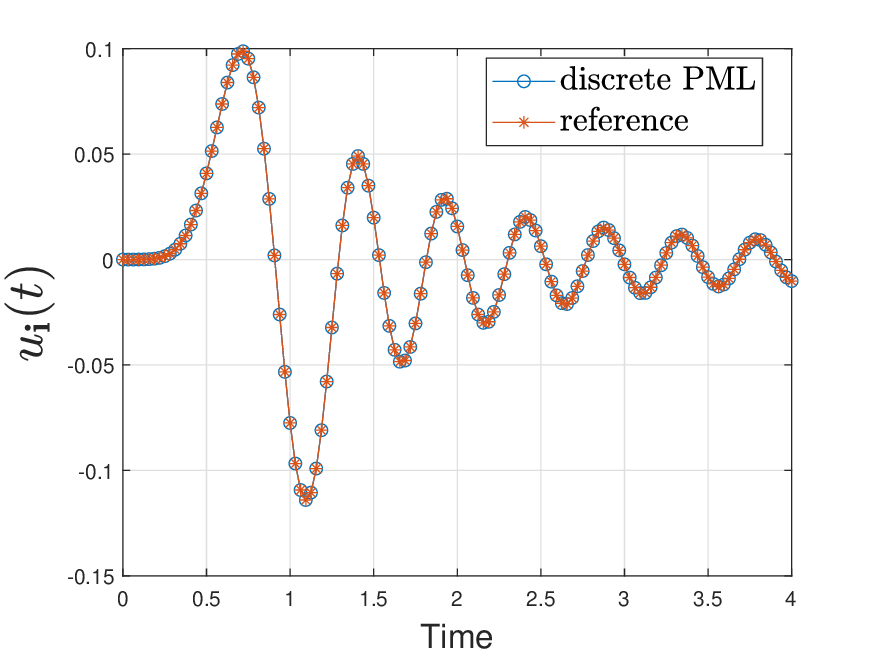}
 	\includegraphics[width=0.4\textwidth]{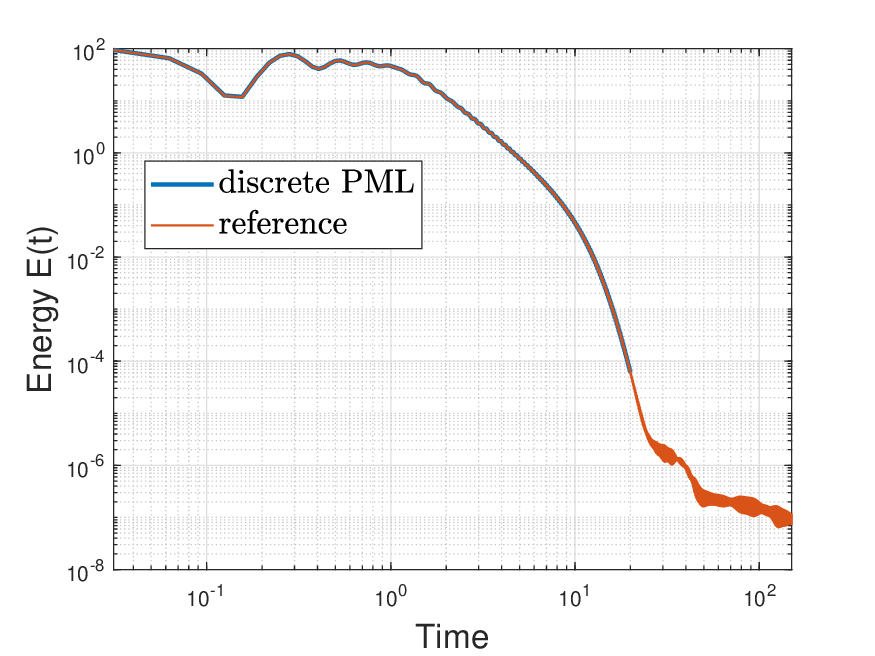}
 	\caption{(Example 2.) Left: Time evolution of the solutions at \( \bx=(0.5,0.5) \). Right: Comparison of the energy evolution in time between the discrete PML solution and reference solution.} \label{fig_ex3stab}
\end{figure}



\section{Conclusion}
This paper presents a derivation of a discrete PML that has no numerical reflection and absorbs scattering waves perfectly. After truncating the domain, both numerical stability and exponentially small residual waves are observed. The discrete PML is for one specific discretized PD scalar wave equation. It provides a general framework for designing a  non-reflecting boundary, regardless of whether the model is local or nonlocal, although a discrete wave equation based on other schemes (such as nonconforming methods) may require different discrete Cauchy-Riemann relations. The current work also relies on taking Fourier transform in time or time-harmonic element wave. In addition, such a treatment may need to be revised for time fractional wave equations and peridynamic elastic waves in unbounded domains.
The stability and effectiveness of the discrete PML are shown numerically, and the corresponding rigorous analysis in the continuous theory would be an important research in the future.

\section*{Acknowledgements}

This work is supported in National Key Research and Development Project of China under grants No. 2024YFA1012600, NSFC under grants No. 12071401, 12171376, and the science and technology innovation Program of Hunan Province 2022RC1191. This work was supported in part by the High Performance Computing Platform of Xiangtan University.

\bibliographystyle{siam}
\bibliography{referrence}

\end{document}